\documentclass[12pt,a4paper,twoside]{article}
\usepackage[ansinew]{inputenc}
\usepackage{amsmath}
\usepackage{amsfonts}
\usepackage{amssymb}
\usepackage{amsthm}
\usepackage{tikz}
\usepackage{ifthen}
 \usetikzlibrary{shapes.geometric,arrows,snakes,backgrounds}

\makeatletter
\def\@seccntformat#1{\csname the#1\endcsname.\quad}
\makeatother

\newcommand{\EE}{{\mathbb{S}^2}}
\newcommand{\RR}{{\mathbb{R}^2}}
\newcommand{\Sd}{{\mathbb{S}^1}}

\newcommand{\R}{\mathbb{R}}
\newcommand{\C}{\mathbb{C}}
\newcommand{\Id}{\mathop{\rm Id}\nolimits}
\newcommand{\Cl}{\mathop{\rm Cl}\nolimits}
\newcommand{\Bd}{\mathop{\rm Bd}\nolimits}
\newcommand{\Inte}{\mathop{\rm Int}\nolimits}

\newcommand{\Sing}{\mathop{\rm Sing}\nolimits}
\newcommand{\diam}{\mathop{\rm diam}\nolimits}

\newcommand{\arccot}{\mathop{\rm arccot}\nolimits}

\newcommand{\0}{\mathbf{0}}
\def\ii{\mathbf{i}}

\newcounter{letra}

\newtheorem{theorem}{Theorem}[section]
\newtheorem{proposition}[theorem]{Proposition}

\newtheorem{lemma}[theorem]{Lemma}
\newtheorem{maintheorem}[letra]{Theorem}
\newtheorem{maincorollary}[letra]{Corollary}
\theoremstyle{definition}
\newtheorem{definition}[theorem]{Definition}
\theoremstyle{remark}
\newtheorem{remark}[theorem]{Remark}

\theoremstyle{plain}

\newcounter{versionfinal}
\ifthenelse{\value{versionfinal}=1}{
\newcommand{\josegines}[1]{}

\newcommand{\corregidooriginal}[2]{
}
\newcommand{\borrar}[1]{
}

\newcommand{\apendices}[1]{
}

}{

\newcommand{\josegines}[1]{\renewcommand{\thefootnote}{\bfseries\color{red}\arabic{footnote}}\footnote{\textcolor{red}{\textbf{Nota de Jose Gines: } #1}}\renewcommand{\thefootnote}{\arabic{footnote}}}

\newcommand{\corregidooriginal}[2]{#1

{\bfseries \color{red} #2}

}

\newcommand{\borrar}[1]{
}

\newcommand{\apendices}[1]{#1}

}

\begin{document}

\title{A topological characterization of the $\omega$-limit sets
of analytic vector fields on open subsets of the sphere}
\author{Jos\'e Gin\'es Esp\'in Buend\'ia and V\'ictor Jim\'enez
Lop\'ez}
\date{\normalsize{Universidad de Murcia (Spain)}\\
\normalsize{\today}}
\maketitle

\begin{abstract}
In \cite{JL}, V. Jim\'enez and J. Llibre characterized, up to
homeomorphism, the $\omega$-limit sets of analytic vector fields
on the sphere and the projective plane. The authors also studied
the same problem for open subsets of these surfaces.

Unfortunately, an essential lemma in their programme for general
surfaces has a gap. Although the proof of this lemma can be
amended in the case of the sphere, the plane, the projective plane
and the projective plane minus one point (and therefore the
characterizations for these surfaces in~\cite{JL} are correct),
the lemma is not generally true, see~\cite{EJQTDS}.

Consequently, the topological characterization for analytic vector
fields on open subsets of the sphere and the projective plane is
still pending. In this paper, we close this problem in the case of
open subsets of the sphere.
\end{abstract}


\smallbreak \noindent{\bf 2010 Mathematics Subject
Classification:} 37E35, 37B99, 37C10.

\section{Introduction and statements of the main results}\label{intro}

In a certain sense, the problem of characterizing, from a
topological point of view, the $\omega$-limit sets of
two-dimensional continuous dynamical systems is as old as the
theory of dynamical systems itself. After all, what the famous
Poincar\'e-Bendixson states, in a present-day formulation, is that
any $\omega$-limit of a sphere flow containing no critical points
is a periodic orbit ---hence, topologically speaking, a Jordan
curve.

If the absence of critical points is no longer required, sphere
$\omega$-limit sets still admit a very clear-cut characterization,
as shown by Vinograd in the early fifties: they are the boundaries
of simply connected regions \cite{VINOGRAD}. Building on
Vinograd's characterization, and some partial results by Smith and
Thomas \cite{SmTh}, V. Jim\'enez L\'opez and G. Soler L\'opez
published a number of papers providing a complete topological
classification of $\omega$-limit sets for continuous flows on all
compact (without boundary) surfaces. These results were summarized
in \cite{BUMI}, where a list of relevant references can also be
found. It is worth emphasizing that, due to a result by
Guti\'errez \cite{Gutierrez}, $C^\infty$-flows are topologically
undistinguishable from continuous flows as far as non-trivial
recurrences do not occur; in particular, Vinograd's theorem is
still true in the smooth realm.

When the very natural assumption of analyticity is added (in fact,
the Poincar\'e-Bendix\-son theorem was first proved by Poincar\'e
for analytic vector fields!), Vinograd's theorem does not work any
more: there are many simply connected sphere regions whose
boundaries (even after topological deformation) cannot be realized
by analytic flows. The topological classification of the
$\omega$-limit sets of analytic flows in the sphere (and also in
the plane and in the projective plane) was accomplished by
Jim\'enez L\'opez and Llibre in \cite{JL}. A similar
classification, now for analytic flows just defined on open
subsets of these surfaces, was also outlined there.

Unfortunately, in doing this, Jim\'enez L\'opez y Llibre made an
oversight, wrongly assuming that, in this setting, an
$\omega$-limit set cannot be locally an arc at any of its points.
While this is true for the sphere, the plane, the projective plane
and the projective plane minus one point (hence the main results
of \cite{JL} remain correct), it is possible, for instance, that
an arc is an $\omega$-limit set for an analytic vector field
defined in the whole sphere except at both endpoints of the arc:
see \cite{EJQTDS}. To make things worse, this unexpected ``arc
issue'' implies that the characterizations of $\omega$-limit sets
for the sphere and the projective plane cannot be more or less
directly extended (as assumed in \cite{JL}) to proper open subsets
of these surfaces, particularly if we intend to preserve
analyticity as much as possible. The aim of this paper is,
therefore, to provide a correct (and optimal) topological
characterization of the $\omega$-limit sets of analytic vector
fields defined on open subsets of the sphere. Our main result is
surprisingly easy to state: these $\omega$-limit sets are,
essentially, the boundaries of simply connected Peano spaces. We
intend to address the similar problem for the projective plane in
a forthcoming paper.

Before stating precisely our results, we need some definitions and
notions. Recall that a function $v=f(u)$, $u=(u_1,\ldots,u_n)$,
mapping an open subset $U$ of $\mathbb{R}^n$ into $\R$, is called
\emph{(real) analytic} if it can be locally written as a
convergent power series in the variables $u_1,\ldots,u_n$. A
function $f:U\to \mathbb{R}^m$ is called \emph{analytic} when each
of its components is analytic in the previous sense.

Throughout the paper, the distance $d(\cdot,\cdot)$ in the unit
sphere $\EE= \{(u_1,u_2,u_3) \in \mathbb{R}^3:u_1^2 + u_2^2 +
u_3^2=1\}$ will remain fixed. We endow $\EE$ with an analytic
differential structure using as charts the \emph{stereographic
projections} $\pi_N:\EE\setminus \{p_N\}\to \RR$ and
$\pi_S:\EE\setminus \{p_S\}\to \RR$ defined, respectively, by
$\pi_N(x,y,z)= (x/(1-z),y/(1-z))$ and $\pi_S(x,y,z)=
(x/(1+z),-y/(1+z))$ (here $p_N=(0,0,1)$ and $p_S=(0,0,-1)$ are the
\emph{north and south poles}). Now, if $f$ is a map from an open
subset $O$ of $\EE$ into $\mathbb{R}^m$, differentiability for $f$
is defined in the usual way. In particular, $f$ is called
\emph{analytic} if the compositions $f\circ \pi_N^{-1}$ and
$f\circ \pi_S^{-1}$ (whenever they make sense) are analytic. If
$f:O\to \mathbb{R}^3$ satisfies that $f(u)$ is tangent to $\EE$ at
$u$ for any $u\in O$, then it is called a \emph{vector field} on
$O$. In this paper we will only deal with $C^\infty$ and analytic
vector fields. We say that a set $A\subset O$ is \emph{analytic
(in $O$)} if it is the set of zeros of some analytic map $F:O\to
\R$.

If $f$ is a  $C^\infty$-vector field on $\EE$, and $u_0\in \EE$,
then we denote by $\Phi_{u_0}(t)$ the maximal solution $u=u(t)$ of
the differential equation $u'=f(u)$ with initial condition
$u(0)=u_0$. The map $u(t)$ is defined for all $t\in \R$, and
$\Phi:\R\times \EE\to \EE$ defined by $\Phi(t,u)=\Phi_u(t)$, the
\emph{flow associated to $f$} is a continuous (in fact, a
$C^\infty$-) map. The set $\Gamma=\Phi_u(\R)$ is called an
\emph{orbit} of $\Phi$, and any subset $\Phi_u(I)$, with $I$ an
interval, a \emph{semi-orbit} of $\Gamma$. If the orbit $\Gamma$
equals $u$ (that is, $f$ vanishes at $u$), then $u$ is called
\emph{singular}, and it is \emph{regular} otherwise. We denote by
$\Sing(\Phi)$ the set of singular points of $\Phi$. The
\emph{$\omega$-limit set} $\omega_\Phi(u)$ of $u$ (or of the orbit
$\Phi_u(\R)$) is the set of accumulation points of $\Phi_u(t)$ as
$t\to \infty$, and the \emph{$\alpha$-limit set} $\alpha_\Phi(u)$
is the analogously defined set for $t\to -\infty$. A set $M\subset
\EE$ is called a \emph{flow box} for $\Phi$ (respectively, a
\emph{semi-flow box} for $\Phi$) if there is a homeomorphism
$h:[-1,1]\times [-1,1]\rightarrow M$ (respectively, a
homeomorphism $h:[-1,1]\times [0,1]\rightarrow M$) such that
$h([-1,1]\times \{s\})$ is a semi-orbit of $\Phi$ for every $s\in
[-1,1]$ (respectively, for every $s\in (0,1]$). In the second
case, we call the arc $h([-1,1]\times \{0\})$ the \emph{border} of
the semi-flow box $M$. The continuity of $\Phi$ implies that,
although the border of a semi-flow box needs not be a semi-orbit
of $\Phi$, it is the union set of some of its semi-orbits. As it
is well known, if $u$ is regular, then it is neighboured by a flow
box.

As has just been said, our aim is to characterize topologically
the $\omega$-limit sets of analytic vector fields $f$ defined on
non-empty open subsets $O$ of $\EE$. To do this we assume that $f$
can be $C^\infty$-extended to the whole $\EE$ by adding singular
points at $\EE\setminus O$: in view of Theorem~\ref{analytic-1}
below, this involves no loss of generality (because it is possible
to multiply $f$ by a positive factor so that the resultant vector
field $\tilde{f}$ has this property, when observe that the
solutions of the differential equation $u'=\tilde{f}(u)$, when
seen as subsets of $O$, are the same as those of $u'=f(u)$). Thus,
when speaking about $\omega$-limit sets of $f$, we are in fact
referring to the $\omega$-limit sets of the flow $\Phi$ associated
to the extension of $\tilde{f}$ to $\EE$. Of course, it is
sufficient to consider the case when $O$ is a \emph{region} (that
is, open and connected), and, as an additional simplification, we
will assume that the complementary of $O$ is \emph{totally
disconnected} (that is, all components of $\EE\setminus O$ are
singletons), because any region of $\EE$ is homeomorphic (in fact,
analytically diffeomorphic) to a region of $\EE$ of this type: see
Proposition~\ref{equiv}(i) and Theorem~\ref{analytic-0}.

A topological space homeomorphic to $[0,1]$, $\R$, the unit
circumference $\mathbb{S}^1=\{z\in \mathbb{C}: |z|=1\}$, the unit
ball $\mathbb{D}^2=\{z\in \mathbb{C}: |z|\leq 1\}$ or $\R^2$ will
be called, respectively, an \emph{arc} (its \emph{endpoints} being
the points mapped by the homeomorphism to $0$ and $1$), an
\emph{open arc}, a \emph{circle}, a \emph{disk}  or an \emph{open
disk}.

We say that a topological space $X$ is \emph{pathwise connected}
if for any $x,y\in X$ there is a continuous map $\varphi:[0,1]\to
X$ such that $\varphi(0)=x$ and $\varphi(1)=y$. Such a map  is
called a \emph{path} (from $x$ to $y$). If additionally, for any
$x,y\in X$, there is an arc in $X$ having $x$ and $y$ as its
endpoints, then $X$ is called \emph{arcwise connected.} When $X$
is Hausdorff, these turn out to be equivalent notions, see
\cite[Corollary~31.6, p. 222]{WILLARD}.

A compact connected Hausdorff space is called a \emph{continuum},
and a locally connected metric continuum is called a \emph{Peano
space}. The Hahn-Mazurkiewicz theorem establishes that a
(non-empty) continuum is a Peano space if and only if it is the
continuous image of the interval $[0,1]$ \cite[Theorem~2, p.
256]{KU}. Hence any Peano space is pathwise (arcwise) connected.
Moreover, it is \emph{locally arcwise connected} as well, that is,
for any $\epsilon>0$ there is $\delta>0$ such that, whenever
$v,w\in X$ and $0<d(v,w)<\delta$, there is an arc  with endpoints
$v,w$ whose diameter is less than $\epsilon$ \cite[Theorem~2, p.
253 and Theorem~1, p. 254]{KU}.


A pathwise connected space $X$  whose fundamental group is trivial
(that is, for any  path $\varphi:[0,1]\to X$ with
$\varphi(0)=\varphi(1)=x$ there is a continuous map $F:[0,1]\times
[0,1]\to X$ such that $F(t,0)=\varphi(t)$ and $F(t,1)=x$ for any
$t$) is called \emph{simply connected}. As shown in
\cite[Proposition~3.2, p. 10]{GREENBERG}), simply connectedness is
equivalent to contractibility: $X$ is said to be
\emph{contractible} if there are $p\in X$ and a  continuous map
$G:[0,1]\times X\to X$ such that $G(0,x)=x$ and $G(1,x)=p$ for any
$x$. It is well known (see, e.g, \cite[Theorem~13.11, p.
274]{RUDIN}) that if  $\emptyset \subsetneqq X\subsetneqq\EE$ is a
region, then $X$ is simply connected if and only $\EE\setminus X$
is connected and if and only if $X$ is an open disk. The
equivalence between simply connectedness of $X$ and connectedness
of $\EE\setminus X$ holds as well when $X\subset \EE$ is a Peano
space, see \cite[Proposition~4.1]{KRRZ}.

Let $E_X$ be the set of points of a Peano space $X$ admitting an
open arc as a  neighbourhood. If $E_X$ is dense in $X$, then $X$
is called a \emph{net}, each component of $E_X$ is called an
\emph{edge} of $X$, and the points of $E_X$ and  $X\setminus
E_{X}$ are respectively called the \emph{edge points} and the
\emph{vertexes} of $X$. Any edge $E$ of a net $X$ is either a
circle (when $E=X$) or an open arc. In this last case $\Cl E$ is
either $E$ plus one vertex of $X$ (and then we get a circle) or
$E$ plus two vertexes of $X$ (and then we get an arc).

\begin{remark}
The previous statements can be proved as follows. Let $E$ be
an edge of $X$, fix $x\in E$, find open arcs $A,B$ in $E$
neighbouring $x$ with $\Cl A\subset B$ and let $p$ and $q$ be the
endpoints of $\Cl A$. Then $X\setminus A$ is trivially locally
connected.

Assume first that this set is connected, hence a Peano space. Then
there is an arc $C\subset X\setminus A$ with endpoints $p$ and
$q$. If $C\subset E$, then, by connectedness, $E$ equals the
circle $C\cup A$. Otherwise, there is a arc $C_p\subset C$ with
endpoints $p$ and $v$, and an arc $C_q\subset C$  with endpoints
$q$ and $w$, such that $v$ and $w$ are vertexes of $X$ and both
$C_p\setminus \{v\}$ and $C_q\setminus \{w\}$ are included in $E$.
Again using the connectedness of $E$, if  $v=w$ then $E\cup \{v\}$
is the circle $C\cup A$, while if $v\neq w$ then $E\cup \{v,w\}$
is the arc $A\cup C_p\cup C_q$.

If $X\setminus A$ is not connected, then it is the union of two
disjoint Peano spaces $V\ni p$ and $W\ni q$. We claim that $V$
(and similarly $W$) is not fully included in $E$. If this is not
true, then any point of $V$ except $p$ disconnects $V$ (otherwise
we could argue as in the above paragraph to find a circle in
$V\subsetneqq E$, then arriving at a contradiction), hence any
pair of points of $V$ disconnect $V$. By \cite[Theorem~2, p.
180]{KU}, $V$ is then a circle and again we get a contradiction.
Thus, there are points in $V$ which do not belong to $E$, and we
can construct an arc $C_p$ with endpoints $p$ and a vertex $v$ in
$V$, such that $C_p\setminus \{v\}\subset E$. Arguing similarly in
$W$ to find a vertex $w\in W$ and an arc $C_q$ with endpoints $q$
and $w$ and such that $C_w\setminus \{w\}\subset E$, we conclude
as before that $E\cup \{v,w\}$ is the arc $A\cup C_p\cup C_q$.
\end{remark}

When the number of edges (and then of vertexes) of a net is finite
it is called a \emph{graph}. If a graph  includes no circles, then
it is called a \emph{tree}; more generally, a Peano space
including no circles is called a \emph{dendrite}. If a tree $X$
has $n$ edges, then it has $n+1$ vertexes: if, moreover, there is
a vertex $c$ belonging to the closure of all its edges, then $X$
is called an ($n$-)star with \emph{center} $c$ and
\emph{endpoints} all other vertexes of $X$. In this scenario, the
edges of $X$ are also said to be the \textsl{branches} of the
star. Strictly speaking, this only makes senses if $n\neq 2$ (if
$n=1$, then $X$ becomes an arc and we can choose as its ``center''
any of its endpoints). Yet, for notational convenience, an arc
will also be referred to as a ``$2$-star'', when all points except
the endpoints are considered to be ``centers'', and get the
corresponding ``edges'' after taking out a ``center'' and both
endpoints. Finally, also by convention, a single point is called a
\emph{$0$-star}, its center being the point itself. If $X$ is a
topological space and $p$ is a point in $X$ neighboured by an
$n$-star with $p$ as its center, then we say that $p$ is a
\textit{star point of $X$ (of order $n$)}, and also an
\emph{$n$-star point of $X$}. Note that the order of a star point
is unambiguously defined and that if $X$ is a net, then $E_X$ is
just the set of $2$-star points of $X$.

We say that $X$ is a \emph{thick arc} if there are pairwise
disjoint subintervals $\{(c_n-\delta_n,c_n+\delta_n)\}_n$ of
$[0,1]$ (with $(c_n)_n$ and $(\delta_n)_n$ sequences in $(0,1)$)
such that $X$ is homeomorphic to $[0,1]\cup \bigcup_n \{z\in
\mathbb{C}:|z-c_n|\leq \delta_n\}$. The points mapped by the
homeomorphism to $0$ and $1$ are, again, called the
\emph{endpoints} of $X$, and are also said to be a  \emph{proper
pair of endpoints} of $X$. In this way we deal with a possible
ambiguity, because the endpoints of a thick arc are not uniquely
defined when $c_n - \delta_n=0$ and/or $c_n + \delta_n=1$ for some
$n$. If the above family of intervals is empty (respectively, only
consists of the interval $(0,1)$), then a thick arc becomes an arc
(respectively, a disk).

The next one is the most important notion of this paper.

\begin{definition}
We say that $\emptyset\subsetneqq A\subsetneqq\EE$ is a
\emph{shrub} if it is a simply connected Peano space.
\end{definition}

Single points, arcs, disks and, in general, thick arcs in $\EE$,
are the simplest examples of shrubs, and dendrites are shrubs as
well \cite[Theorem~3, p. 375 and Corollary~7, p. 378]{KU}. If $A$
is a shrub, then each disk in $A$ which is not included in a
larger disk is called a \emph{leaf (of $A$)}. If all points of an
arc in $A$, except its endpoints, are $2$-star points of $A$, then
it is called a \emph{sprig}. A point of a shrub $A$ may be an
\emph{interior point} (if it belongs to $\Inte A$), an
\emph{exterior point} (if it belongs to the boundary of a leaf and
does not disconnect $A$), a \emph{sprig point} (if it is a point,
but not an endpoint, of a sprig), or, otherwise, an \emph{bud}. If
a bud disconnects $A$, then it is called a \emph{node}, and
otherwise a \emph{tip}. Therefore, nodes and sprig points are the
of $A$ disconnecting it. See Figure~\ref{figura1}. Clearly, the
set of buds of $A$ is closed. A thick arc $B$ in $A$ is called a
\emph{stem} when, whenever an interior point, an exterior point or
a sprig point belongs to $B$, the corresponding leaf or sprig is
included in $B$. If a shrub $A$ is a union of finitely many
leaves, then it is called a \emph{cactus}. If  $A$ is the union of
a cactus $D$ and $m$ sprigs, all of them having some endpoint in
$D$, we call $A$ an \emph{$m$-prickly cactus}.

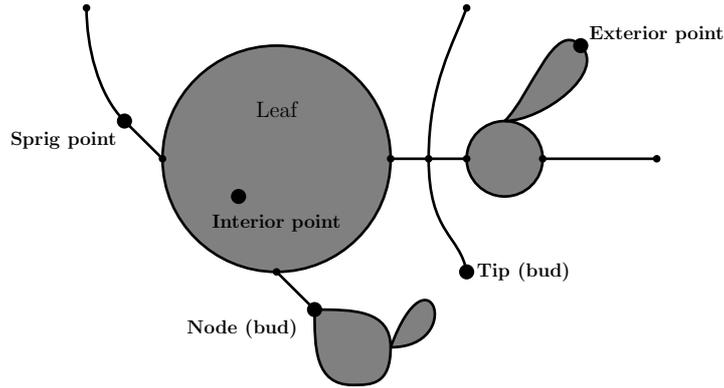
\begin{figure}
    \centering
\begin{tikzpicture}[scale=0.5]

 \filldraw[line width=1, fill=gray] (5,6) circle (3);
 \draw[line width=1] (8,6) -- (10,6);
 \filldraw[line width=1, fill=gray] (11,6) circle (1);
 \draw[line width=1] (12,6) -- (15,6);
\fill (8,6) circle (0.1); \fill (10,6) circle (0.1); \fill (12,6)
circle (0.1); \fill (15,6) circle (0.1);
 \draw[line width=1]
     (0,10) .. controls +(270:1) and +(135:1) ..  (1,7)
                .. controls +(315:1) and +(135:1) .. (2,6);
\fill (0,10) circle (0.1); \fill (1,7) circle (0.2); \fill (2,6)
circle (0.1);
 \draw[line width=1] (5,3) -- (6,2);

 \filldraw[line width=1, fill=gray]
     (6,2) .. controls +(270:1) and +(180:1) ..  (7,0)
           .. controls +(0:0.5) and +(270:1) .. (8,1)
           .. controls +(90:1) and +(0:1) .. (6,2);
 \fill (6,2) circle (0.2);
 \filldraw[line width=1, fill=gray]
     (8,1) .. controls +(0:2) and +(70:3) ..  (8,1);
 \filldraw[line width=1, fill=gray]
     (11,7) .. controls +(45:1) and +(135:1) ..  (13,9)
            .. controls +(315:1) and +(0:1) ..  (11,7);
 \fill (13,9) circle (0.2);
 \draw[line width=1]
     (10,3) .. controls +(100:1) and +(270:2) ..  (9,6)
            .. controls +(90:2) and +(250:1) ..  (10,10);
\fill (10,3) circle (0.2); \fill (5,3) circle (0.1); \fill (9,6)
circle (0.1); \fill (10,10) circle (0.1);


\node[scale=0.6] at (11.5,3) {\textbf{Tip (bud)}};
\node[scale=0.7] at (5,7.3) {Leaf}; \fill (4,5.0) circle (0.2);
\node[scale=0.6] at (5,4.3) {\textbf{Interior point}};
\node[scale=0.6] at (4.1,1.5) {\textbf{Node (bud)}};
\node[scale=0.6] at (-0.6,6.5) {\textbf{Sprig point}};
\node[scale=0.6] at (15.0,9.3) {\textbf{Exterior point}};

\end{tikzpicture}
\caption{\label{figura1} The different parts of a shrub.}
\end{figure}

If $A$ is a shrub, then all components $\{R_j\}_j$ of $\Inte A$
are open disks (because $R=\EE\setminus A$ is connected, hence
$\EE\setminus R_j= R\cup \Bd A\cup \bigcup_{j'\neq j} R_{j'}$ is
connected as well). If fact, more is true: their closures $D_j=\Cl
R_j$ are disks. Therefore, the leaves of $A$ are exactly the disks
$D_j$. As some consequences, any circle in $A$ must be included in
one of its leaves, distinct leaves of a shrub can have at most one
common point, and if shrub has no leaves, then is a dendrite.

\begin{remark}
To prove that $D_j$ is a disk it is enough to show, by
\cite[Remark~14.20(a), p. 291]{RUDIN}, that if a sequence
$(u_n)_{n=1}^\infty$ of points in $R_j$ converges to a point $u\in
\Bd R_j$, then there is a path in $D_j$ monotonically passing
through the points $u_n$ and ending at $u$ (that this, there is a
continuous map $\varphi:[0,1]\to D_j$ and numbers $0\leq
t_1<t_2<\cdots$, $t_n\to 1$,  such that $\varphi(t_n)=u_n$ for all
$n$ and $\varphi([0,1))\subset R_j$). This last statement is a
direct consequence of the following fact: if $\epsilon>0$, then
there is $\delta>0$ such that, whenever $v,w\in R_j$ and
$0<d(v,w)<\delta$, we can find  an arc in $R_j$ with endpoints
$v,w$ whose diameter is less than $\epsilon$. Certainly, such a
$\delta>0$ exists, due to the local arcwise connectedness of $A$,
except than we cannot guarantee that the small arc connecting $v$
and $w$, call it $L$, is fully contained in $R_j$. One thing, at
least, is sure: $L\subset D_j$. Otherwise, we could easily
construct a circle $C\subset A$ intersecting both $R_j$ and
$A\setminus D_j$, and simply connectedness forces that one of the
open disks enclosed by $C$ is included in $A$, which contradicts
that $R_j$ is a component of $\Inte A$. Thus, $L\subset D_j$ and,
similarly as above, we can construct a circle $C'$ in $D_j$
including all points of $L\cap \Bd R_j$. Since $C'$ encloses an
open disk fully included in $D_j$ (thus, indeed, in $R_j$), it is
easy to modify slightly $L$ so that the resultant arc $L'$ still
has diameter less than $\epsilon$, and $v$ and $w$ as its
endpoints, and additionally satisfies $L'\subset R_j$.
\end{remark}

\begin{definition}
We say that a shrub $A$ is \emph{realizable} if its set of buds is
totally disconnected.
\end{definition}

\begin{remark}\label{shrubsRegular}
If $A$ is a shrub, then $\Bd A$ is a Peano space \cite[Theorem~4,
p. 512]{KU}. Therefore, if $A$ is a realizable shrub, then $\Bd A$
is a net, the buds of of $A$ are the vertexes of $\Bd A$, and the
exterior and sprig points of $A$ are the edge points of $\Bd A$.
\end{remark}

\begin{definition}
Let $A$ be a shrub.
 \begin{itemize}
 \item We say that a bud $u$ of $A$ is \emph{odd}
if either $u$ is not a star point of $\Bd A$ or $u$ is in no leaf
of $A$ and, for some odd positive integer $n$, $u$ is a star point
of $\Bd A$ of order $n$.
 \item Let $K$ be a maximal connected union of leaves of $A$. We say that $K$
is an \emph{odd cactus (of $A$)} if there is an $n$-prickly cactus
neighbouring $K$ in $A$  for some odd number $n$.
 \end{itemize}
\end{definition}

\begin{remark}
Clearly, the set of odd buds of a shrub is closed, and a set
consisting of all odd buds of a shrub and one point from each of
its odd cactuses, is closed (and totally disconnected if the shrub
is realizable) as well. On the other hand, observe that if the set
of odd buds of a shrub is totally disconnected, then it is
realizable.
\end{remark}

We are ready to state our main results:

\begin{maintheorem}
 \label{teoA}
Let $O\subset \EE$ be a region such that $T=\EE\setminus O$ is
totally disconnected. Let $f$ be a $C^\infty$-vector field on
$\EE$ which is analytic on $\EE\setminus T$. Then any
$\omega$-limit set of $f$ is the boundary of a shrub. Moreover,
all odd buds of the shrub are contained in $T$ (hence it is
realizable) and every odd cactus of the shrub must intersect $T$.
\end{maintheorem}

Conversely, we have:

\begin{maintheorem}
 \label{teoB}
Let $A\subset\EE$ be a shrub and let $T\subset A$ contain all odd
buds of $A$ and one point from each of the odd cactuses of $A$.
Then there are a homeomorphism $h:\EE\to\EE$, and a
$C^\infty$-vector field on $\EE$, analytic on $h(\EE\setminus T)$,
having the boundary of $h(A)$ as an $\omega$-limit set.
\end{maintheorem}

\begin{remark}
In our proof of Theorem~\ref{teoB} we do not care whether the
points of the $\omega$-limit set are singular or regular for the
corresponding flow. The much more difficult problem of
constructing $\omega$-limit sets with prescribed sets of regular
and singular points will not be considered here. Still, note that
all sprig points of the shrub must be singular (this is a
consequence of Lemma~\ref{teoA-2}).
\end{remark}

\begin{remark}
 Observe that Theorem~\ref{teoB} is a kind of ``strong'' converse
 of Theorem~\ref{teoA} because $A$ needs not be
 realizable. For instance, enumerate the set of rationals in
 $[0,1]$ as $\{\frac{p_n}{q_n}\}_{n=1}^\infty$ (written as
 irreducible fractions) and let  $A'\subset \R^2$ be the dendrite
 consisting of the union of the arcs $B'=\{(x,0):x\in [0,1]\}$ and
 $B_n'=\{(\frac{p_n}{q_n},y):y\in [0,\frac{1}{q_n}]\}$,
 $n=1,2,\ldots$. Use the (inverse of the) stereographic projection
 to get the dendrite  $A$ in $\EE$ which is the union of the
 corresponding arcs $B$ and $\{B_n\}_{n=1}^\infty$, and realize
 that the set $T$ of (odd) buds in $A$ consists of the whole arc $B$ and
 all endpoints of the arcs $B_n$. From the proof of Theorem~\ref{teoB} it follows
 that there is an analytic vector field on $\EE\setminus
 T$ having $A$ as an $\omega$-limit set.
\end{remark}

\begin{maincorollary}
 \label{coroC}
Up to homeomorphisms, a set is an $\omega$-limit set of some
analytic vector field defined on $\EE$ except for a totally
disconnected complementary if and only if it is the boundary of a
realizable shrub.
\end{maincorollary}

The paper is organized as follows. In Sections~\ref{secttop} and
\ref{sectana} we summarize a number of topological and analytical
results, mostly well known, which will be needed later.
Theorem~\ref{teoA} is proved in Section~\ref{proofA}. An
intermediate result, fundamental for the proof of
Theorem~\ref{teoB}, is shown in Section~\ref{proofint}. Then we
proceed to prove Theorem~\ref{teoB} in Section~\ref{proofB}.

\section{On the sphere homeomorphisms and the topology of shrubs}
 \label{secttop}

Throughout this paper, several intuitive (yet deep) topological
results from the topology of the sphere will be needed: in this
regard, an old but outstanding reference is \cite{KU}, and we will
cite it quite often. Among these results, the following ones may
not be as well known as the Jordan curve theorem, but they will be
implicitly used a number of times: if $K\subset \EE$ is compact
and totally disconnected, then there is an arc in $\EE$ including
$V$ \cite[Theorem~5, p. 539 (see also p. 189)]{KU}; if $B$ and
$B'$ are either arcs or circles in $\EE$, then there is a
homeomorphism $h:\EE\to \EE$ mapping $B$ onto $B'$
\cite[Corollary~2, p. 535]{KU}. A simple consequence of this: if
$O\subset \EE$ is open and $\EE\setminus O$ is totally
disconnected, then $O$ is a region.

The extension result concerning arcs and circles we have just
mentioned is a special case of a problem with a long tradition in
the literature (see the references in \cite{Adkisson3,
Knobelauch}): the study of conditions under which a homeomorphism
between two subsets of a manifold $M$ can be extended to a
homeomorphism of $M$ onto itself.  In particular, the cases when
$M$ is the plane or the sphere have been investigated in great
depth: \cite{Adkisson1, Adkisson2, Adkisson4, Gehman1, Gehman2}
and \cite[Section~61.V]{KU}. For instance, in \cite{Adkisson2},
necessary and sufficient conditions characterizing when two Peano
spaces $X_1,X_2 \subset \EE$ are \emph{compatible}, that is, there
is a sphere homeomorphism mapping $X_1$ onto $X_2$, are given. We
next explain the main result of \cite{Adkisson2}.

Let $A,B,C\subset \EE$ be three arcs having exactly one common
endpoint, and no other intersection point (hence $A\cup B\cup C$
is a $3$-star in $\EE$): then we say that $(A,B,C)$ is a
\textit{triod}. Two triods $(A_1,B_1,C_1)$ and $(A_2,B_2,C_2)$ are
said to have the \textit{same sense} if there is a homeomorphism
$f:\EE \to \EE$, homotopic to the identity, such that
$f(A_1)=A_2$, $f(B_1)=B_2$ and $f(C_1)=C_2$. (As usual, two
continuous maps $f,g:\EE\to \EE$ are said to be \emph{homotopic}
if there is a continuous map $H:[0,1]\times \EE\to \EE$ such that
$H(0,u)=f(u)$ and $H(1,u)=g(u)$ for any $u\in \EE$.) Two triods
which do not have the same sense are said to have \textit{opposite
sense}.

\begin{remark}
Every homeomorphism $f:\EE \to \EE$ is homotopic either to the
identity or to the antipodal map $a(u)=-u$. Indeed, any continuous
map $f:\EE \to \EE$ induces a homomorphism $f_{*}:H_2(\EE) \to
H_2(\EE)$, where $H_2(\EE)$ is the second homology group of $\EE$
\cite[pp. 110--111]{Hatcher}. Since $H_2(\EE)$, as a group, is
isomorphic to $(\mathbb{Z},+)$ \cite[Corollary~2.14, p.
114]{Hatcher}, there exists an integer $\deg(f)$, the so-called
\emph{degree} of $f$, such that $f_*(n)= n\deg(f)$ for every $n
\in \mathbb{Z}$; for example, the identity map has degree $1$
while the antipodal map has degree $-1$ \cite[p. 134]{Hatcher}.
Two homotopic continuous maps $f,g:\EE \to \EE$ have the same
degree \cite[Theorem~2.10, p. 120]{Hatcher} and, reciprocally, two
continuous maps $f,g:\EE \to \EE$ with the same degree must be
homotopic (this is a deep result proved by Hopf, see
\cite[Theorem~4.15, p. 361]{Hatcher}). To conclude, observe that
if $f$ is a homeomorphism, then the induced homomorphism $f_*$ is
a group isomorphism and $\deg(f)$ equals $-1$ or $1$. (In general,
given  two continuous maps $f,g: \EE \to \EE$, $(f \circ
g)_{*}=f_{*} \circ g_{*}$ and $\deg(f \circ g) = \deg(f) \deg(g)$
\cite[p. 134]{Hatcher}.)
\end{remark}

A simpler way to know when two triods have the same, or opposite
sense, is as follows. Assume that the common point $p$ of the arcs
of a triod $(A,B,C)$ is not the north pole $p_N$ and write
$\pi_N(A)=A'$, $\pi_N(B)=B'$, $\pi_N(C)=C'$ and $\pi_N(p)=p'$.
Then we say that $(A,B,C)$  is \emph{positive} when, after taking
an open euclidean ball $U$ of center $p'$ and radius $\epsilon>0$
small enough, there is $\theta_0\in \mathbb{R}$ such that the
first intersection points of the arcs $A',B',C'$ with $\Bd U$ can
be respectively written as $p'+\epsilon e^{\ii\theta_A},
p'+\epsilon e^{\ii\theta_B}, p'+\epsilon e^{\ii\theta_C}$, with
$\theta_0=\theta_A<\theta_B<\theta_C<\theta_0+2\pi$.  We say that
the triod is \emph{negative} when it is not positive. If $p=p_N$,
then we say that $(A,B,C)$ is \emph{positive} (respectively,
\emph{negative}) if $(a(A),a(B),a(C))$ is negative (respectively,
positive), with $a$ being the antipodal map. As it turns out (see
\cite[Theorem~8 and 9]{Adkisson2} and \cite{KLINE}), two triods
have the same sense if and only if they have the same ``sign'',
that is, they are both positive or both negative. As a
consequence, observe that if $(A,B,C)$ is an arbitrary triod in
$\EE$, including or not the north pole, then $(A,B,C)$ and
$(f(A),f(B),f(C))$ have opposite sense when $f$ is the antipodal
map and, in general, when $f$ is not homotopic to the identity
(because $a\circ f^{-1}$ is homotopic to the identity so
$(f(A),f(B),f(C))$ and $(a(A),a(B),a(C))$ have the same sense). In
other words, if $f:\EE \to \EE$ is a homeomorphism, then $(A,B,C)$
and $(f(A),f(B),f(C))$ have the same sense if and only $f$ is
homotopic to the identity.

Let $g:X_1 \to X_2$ be a homeomorphism between two Peano spaces
$X_1,X_2 \subset \EE$. We say that $g$ \emph{preserves senses}
(respectively, \emph{reverses senses}) if a triod $(A_1,B_1,C_1)$
in $X_1$ is positive if and only if the triod
$(g(A_1),g(B_1),g(C_1)$ in $X_2$ is positive (respectively,
negative). More generally, we say that $g$ \emph{preserves the
geometrical configuration} if either it preserves senses, or
reverses senses. We are now ready to state the main result of
\cite{Adkisson2}:

\begin{theorem}
 \label{exthomeo}
Let $X_1,X_2\subset \EE$ be Peano spaces. Then they are compatible
if and only if there is a homeomorphism $g:X_1\to X_2$ preserving
the geometrical configuration. Moreover, in this case there exists
a homeomorphism $h:\EE \to \EE$ such that $h(u)=g(u)$ for every $u
\in X_1$.
\end{theorem}

%
%

Sometimes it is useful to see some subsets of $\EE$ as single
points without losing the topological structure of $\EE$. The
following result explains how to do it:

\begin{proposition}
\label{equiv} Let $\{C_i\}_i$ be a family of pairwise disjoint
continua in $\EE$. Assume that $\EE\setminus C_i$ is connected for
any $i$ and, additionally, that one of the following conditions is
satisfied:
 \begin{itemize}
 \item[(i)] There is an open set $O$ such that $\{C_i\}_i$ is the
 family of connected components of $\EE\setminus O$.
 \item[(ii)] The family $\{C_i\}_i$ is countable and (if infinite)
 the diameters of the sets $C_i$ tend to zero.
 \end{itemize}
Then, after defining the equivalence relation $\sim$ in $\EE$ by
$x\sim y$ if either $x=y$ or there is $i$ such that both $x$ and
$y$ belong to $C_i$, the quotient space $\Sigma:=\EE/\sim$ is
homeomorphic to $\EE$.
\end{proposition}

\begin{proof}
Let $\Pi:\EE\to \Sigma$ be the projection map, when recall that
$\mathcal{U}$ is open in $\Sigma$ if and only if
$\Pi^{-1}(\mathcal{U})$ is open in $\EE$. In view of
\cite[Theorem~8, p. 533]{KU} we are left to show:
 \begin{itemize}
 \item[(*)] $\Sigma$ is Hausdorff;
 \item[(**)] $\Sigma\setminus \{X\}$ is connected for any $X\in \Sigma$.
 \item[(***)] $\Pi^{-1}(\mathcal{C})$ is connected for any
 connected set  $\mathcal{C}\subset \Sigma$.
 \end{itemize}

To prove (*) we assume first that (i) holds. Let $X,Y\in \Sigma$,
$X\neq Y$. We must find disjoint open neighbourhoods
$\mathcal{U}(X)$ and $\mathcal{U}(Y)$ of $X$ and $Y$ in $\Sigma$.
If $X$ and/or $Y$ is a point from $O$ this is trivial because $O$
is open, so assume that both $X$ and $Y$ are components of
$\EE\setminus O$. Since $O$ is open, \cite[Theorem~2, p. 169]{KU}
implies that $X$ is the intersection of all open closed sets of
$\EE\setminus O$ (with respect to the topology of $\EE\setminus
O$) including it. In particular, it is possible to find disjoint
compact sets $A,B$ with $A\cup B=\EE\setminus O$, $X\subset A$,
$Y\cap B\neq \emptyset$. The connectedness of any set $C_i$
implies that either $C_i\subset A$ or $C_i\subset B$. Therefore,
$Y\subset B$. Find pairwise disjoint open sets $V\supset A$ and
$W\supset B$. Since, for any $i$, either $C_i\subset V$ or
$C_i\subset W$, we get that $\mathcal{U}(X)=\Pi(V)$ and
$\mathcal{U}(Y)=\Pi(W)$ are the neighbourhoods we are looking for.

Now we prove (*) assuming that (ii) holds. Given $X,Y\in \Sigma$,
$X\neq Y$, we first find disjoint open sets $V,W$ in $\EE$ with
$X\subset V$, $X\subset W$. Realize that the resultant set $V'$
after removing from $V$ the points from the components $C_i$ such
that $C_i\cap \Bd V\neq\emptyset$ is also open (here we need that
the diameters of the sets $C_i$ go to zero), and the same is true
for the analogously defined set $W'$. Then
$\mathcal{U}(X)=\Pi(V')$ and $\mathcal{U}(Y)=\Pi(W')$ are disjoint
open neighbourhoods of $X$ and $Y$ in $\Sigma$.

Statement (**) is immediate: since $\EE\setminus X$ is connected
by hypothesis, and $\Pi$ is continuous, $\Pi(\EE\setminus
X)=\Sigma\setminus \{X\}$ is connected as well.

Note finally that $\Pi$ is a closed map by (*). Then (***) follows
from \cite[Theorem~9, p. 131]{KU} and the fact that any $X\in
\Sigma$ is a connected subset of $\EE$.
\end{proof}

If $A$ is a shrub, then, as formerly said, $\Bd A$ is a Peano
space. This implies that if the family of leaves of $A$ is
infinite, then their diameters tend to zero \cite[Theorem~10, p.
515]{KU}, and if $u,v\in A$ are buds or exterior points of $A$,
then there is exactly one stem having them as a proper pair of
endpoints (to find such a stem, start from an arc $B$ in $A$
having $u$ and $v$ as its endpoints, add to it all the leaves
whose interiors are intersected by $B$, and realize that, due to
the simple connectedness, the intersection of each such leaf with
$B$ is a subarc of $B$). In this way, we extend the main property
of dendrites (any two points of a dendrite are joined by a unique
arc) to shrubs, which could be informally described as ``thick
dendrites''.


We say that a sequence $(S_k)_{k=1}^{n_0}$ of thick arcs in $\EE$,
$n_0\leq \infty$, is a \emph{skeleton} if for any $1\leq k\leq
n_0-1$ the thick arc $S_{k+1}$ intersect $\bigcup_{i=1}^k S_i$ at
exactly one endpoint of $S_{k+1}$. Observe that, in the finite
case $n_0<\infty$, by the Janiszewski theorem \cite[Theorem~7, p.
507]{KU}, the union set $\bigcup_{k=1}^{n_0} S_k$ is a shrub. A
simple consequence of separability and the ``dendrite-like''
structure of shrubs is that for any shrub $A$ there is a (not
necessarily unique) skeleton of stems whose union set is dense in
$A$: we call it a \emph{skeleton of $A$}. Observe that if $U$ is
this union set and $u\in A$ does not belong to $U$, then $u$ must
be a bud. Moreover, if $u$ belongs to a stem $B$ of $A$, then it
must be one of the endpoints of $B$. Otherwise we could assume,
due to the density of $U$ in $A$ and the local arcwise
connectedness of $A$, that there is $k$ such that $A_k$ contains a
pair of proper endpoints $v,w$ of $B$, which is impossible because
there would be two distinct stems having $v,w$ as a proper pair of
endpoints: one included in $A_k$ and $B$. As a consequence, $u$ is
a tip.

If, on the other hand, the closure of the union set of the thick
arcs of a skeleton $\Psi=(S_k)_{k=1}^{n_0}$ is a shrub $A$, then
we say that $\Psi$ is \emph{extensible} to $A$. As we have just
emphasized, any finite skeleton is extensible. The next
proposition deals with the infinite case.

\begin{proposition}
 \label{skeleton1}
Let $(S_k)_{k=1}^\infty$ be an infinite skeleton, write
$A_k=\bigcup_{i=1}^k S_i$, and assume that $\diam(A_k)\leq 1/2^k$
and there are circles $C_k$ (disjoint from all shrubs $A_{k'}$)
such that the region $R_k$ of $\EE\setminus C_k$ not intersecting
the shrubs contains all points $u\in \EE$ satisfying $d(u,A_k)\geq
1/2^k$. Then $(S_k)_{k=1}^\infty$ is extensible.
\end{proposition}

\begin{proof}
The hypothesis on the diameters of the sets $A_k$ allows us to
construct, inductively, continuous onto maps $\varphi_k:[0,1]\to
A_k$ such that, if $k'\geq k$, then
$d(\varphi_k(t),\varphi_{k'}(t))< 1/2^{k-1}$ for all $t\in [0,1]$.
Therefore, the maps $\varphi_k$ converge uniformly to a continuous
map $\varphi:[0,1]\to \EE$, whose image $A=\varphi([0,1])$ is a
Peano space. Clearly, $\EE\setminus A$ is the region
$\bigcup_{k=1}^\infty R_k$. Hence $A$ is a shrub.
\end{proof}

\begin{remark}
 \label{skeleton0}
Regarding Proposition~\ref{skeleton1}, the following consequence
of Proposition~\ref{equiv} will be useful in
Subsection~\ref{casosimple}: if $A$ is a shrub and $\epsilon>0$,
then there is a circle $C$, disjoint from $A$, such that  the
region $R$ of $\EE\setminus C$ not intersecting $A$ contains all
points $u\in \EE$ satisfying $d(u,A)\geq \epsilon$.
\end{remark}

Let $A,A'$ be shrubs and assume that $\Psi=(S_k)_{k=1}^\infty$ and
$\Psi'=(S_k')_{k=1}^{n_0}$ are infinite skeletons of $A$ and $A'$.
Also, let $A_k=\bigcup_{i=1}^k S_i$ and $A_k'=\bigcup_{i=1}^k
S_i'$ for every $k$. We say that $\Psi$ and $\Psi'$ are
\emph{compatible} if there is a sequence of homeomorphisms
$h_k:A_k\to A_k'$ preserving the geometrical configuration and
such that each map $h_{k+1}$ extends $h_{k}$.

\begin{proposition}
 \label{skeleton2}
Assume that two shrubs $A$ and $A'$ admit compatible skeletons
$\Psi$ and $\Psi'$. Then $A$ and $A'$ are compatible.
\end{proposition}

\begin{proof}
For any $k$, let $h_k:A_k\to A_k'$ be as before and denote by $U$
and $U'$ the union sets of the stems of $\Psi$ and $\Psi'$. Note
that, since each homeomorphism $h_{k+1}$ extends $h_k$, either of
all them preserve senses or all of them reverse senses. We may
assume that the first case holds.

Extend continuously the maps $h_k$ to maps $f_k:A\to A'$ as
follows: if $u\in A\setminus A_k$, and $B$ is an arc with endpoint
$u$ and intersecting $A_k$ exactly at its other endpoint point
$v$, then we define $f_k(u)=h_k(v)$. The simple and local arcwise
connectedness of $A'$, and the density of $U'$, imply that, for
any $\epsilon>0$, there is $k=k_\epsilon$ such that  the diameter
of any arc connecting a point from $A\setminus A_k'$ to $A_k'$
must be less than $\epsilon$. From this, and again  the simple
connectedness of $A'$, $d(f_k(u),f_{k'}(u))<\epsilon$ for any
$k'\geq k$ and $u\in A$. Then the maps $f_k$ converge uniformly to
a continuous map $f:A\to A'$, and we can extend similarly the maps
$h_k^{-1}$ to maps $g_k:A'\to A$ which converge uniformly to a map
$g:A'\to A$. Since $g\circ f$ and $f\circ g$ map identically $U$
and $U'$ into themselves, and these sets are dense, $f$ is a
homeomorphism with inverse $g$.

To finish the proof we must show that $f$ preserves senses. Let
$u$ be the common point of the arcs of an arbitrary triod
$(K,L,M)$ in $A$. Then neither $u$ nor the other points from the
arcs (except maybe their endpoints) is a tip, which implies, using
if necessary some slightly smaller arcs, that $K\cup L\cup
M\subset U$. This means, in fact, that $K\cup L\cup M\subset A_k$
for some $k$. Then $(K,L,M)$ and
$(f(K),f(L),f(M))=(h_k(K),h_k(L),h_k(M))$ have the same sense.
\end{proof}

\section{Some useful results on analyticity}
 \label{sectana}

In the introduction we endowed $\EE$ with an analytic differential
structure via the stereographic projections $\pi_N$ and $\pi_S$.
Not that this really matters: this analytic structure is unique up
to analytic diffeomorphisms. In fact, a much stronger result,
following from \cite[Theorem~3]{GRAUERT},
\cite[Corollary~1.18]{WHITEHEAD} and \cite[Theorem~1.4]{HUMO}
after using as a modulus (in the terminology of \cite{HUMO}) the
distance map to the set $C=g^{-1}(C')$, holds:

\begin{theorem}
 \label{analytic-0}
If $S$ and $S'$ are analytic surfaces, $g:S\to S'$ is a continuous
onto map, and $C'$ is a closed subset of $S'$ such that the
restriction of $g$ to $\Omega=S\setminus C$ is a homeomorphism
between $\Omega$ and $\Omega'=S'\setminus C'$, then there is a
continuous onto map $h:S\to S'$ such that $h(u)=g(u)$ for any
$u\in C$ and the restriction of $h$ to $\Omega$ is an analytic
diffeomorphism between $\Omega$ and $\Omega'$. In particular, if
$S$ and $S'$ are homeomorphic, then they are analytically
diffeomorphic.
\end{theorem}

Here, by a \emph{surface} we mean a Hausdorff, second countable,
topological space which is locally homeomorphic to $\RR$. Note
that a surface needs not be compact.  An \emph{analytic surface}
is a surface equipped with an analytic differential structure.
Needless to say, a diffeormorphism between two analytic surfaces
is called \emph{analytic} when it is analytic after being locally
transported to the plane by the charts. Analytic maps from an open
subset $U$ of an analytic surface to $\R^m$, and analytic sets in
$U$, are defined in the obvious way.

For instance, if we endow $\mathbb{R}_\infty^2$, the one-point
compactification of $\RR$, with an analytic structure using as
charts the identity in $\RR$ and the inversion map
$z\rightarrowtail 1/z$ (here, of course, we are identifying $\RR$
and $\mathbb{C}$ via $(x,y)\leftrightarrow x+iy$ and meaning
$1/\infty=0$), then the extended stereographic projections
(writing $\pi_N(p_N)=\pi_S(p_S)=\infty$) provide analytic
diffeomorphisms between $\EE$ and $\mathbb{R}_\infty^2$.

In Theorems~\ref{analytic-1}-\ref{analytic-4} below, $O$ is an
open subset of $\EE$.  The first of them, a consequence of
\cite[Lemma~6]{WHITNEY}, allows us to extend local analytic vector
fields to the whole sphere still retaining $C^\infty$-regularity.

\begin{theorem}
 \label{analytic-1}
Let  $f:O\to \R^n$ be an analytic map. Then there is an analytic
map $\rho:\EE\to (0,\infty)$ such that $\rho f$ (after being
extended as zero outside $O$) is $C^\infty$ in the whole $\EE$.
\end{theorem}

The contents of the following theorem  are classical and well
known, see, e.g., \cite[Theorem~4.3]{JL}, except maybe for the
``parity'' statement, which is due to Sullivan \cite{Sullivan};
alternatively, check \cite{EJSullivan} for a recent
``dynamically based'' proof.

\begin{theorem}
 \label{analytic-2}
Assume that $O \subset \EE$ is a region and let $A\subset O$ be
analytic. Then either $A=O$ or every $u \in A$ is a star point in
$A$ of even order. If, moreover, the corresponding star
neigbouring $u$ is small enough, then, after removing  its center
$u$ and its  endpoints, the resultant open arcs admit smooth (in
fact, analytic) parametrizations.
\end{theorem}

Later in the paper we will consider unions of analytic sets in
open subsets of the sphere. In general, the union of an arbitrary
family of analytic sets in $O$ may not be analytic. It is easy to
use Theorem~\ref{analytic-2} above to find a counterexample; for
instance, the union of an infinite countable family of circles in
$\R^2$ which pairwise meet in the origin cannot be analytic (and
such a family of circles can be easily chosen with all the circles
being analytic sets). Nevertheless, the following result is proved
in \cite[p. 154]{BRUHAT} (a family of subsets of a topological
space $X$ is said to be \textit{locally finite} if every point of
$X$ possesses a neighbourhood which only meets finitely many
subsets of the family):

\begin{theorem}\label{bruhat} If $\mathcal{F}$
is a locally finite family of analytic sets in $O$, then the union
of the sets from $\mathcal{F}$ is also an analytic set in $O$.
\end{theorem}

%

Our last theorem, establishing the local structure of
$\omega$-limit sets for analytic vector fields, was proved in
\cite[Lemma~4.6 (see also the comment below Remark~4.7)]{JL}.

\begin{theorem}
 \label{analytic-4}
Let $f:O\to \R^3$ be an analytic vector field and let $\Phi$ be
the flow associated to $f$. Let $p\in O$ and assume that
$\Omega=\omega_\Phi(p)$ is not a singleton. Let $u\in \Omega\cap
O$. Then there are a disk $D\subset O$ neighbouring $u$ and an
$m$-star $X$, $m\geq 2$, having $u$ as its center, such that
$X=\Omega\cap D$ and $X$ intersects $\Bd D$ exactly at its
endpoints. Moreover, if $Q$ is any of the components of
$D\setminus X$, then either the orbit $\Gamma=\Phi_p(\R)$ does not
 intersect $Q$, or $\Cl Q$ is a semi-flow box
(intersecting $X$ at its border $B$) and $\Gamma$ accumulates at
$B$ from $Q$.
\end{theorem}

\section{Proof of Theorem~\ref{teoA}}
 \label{proofA}

Let $\Phi$ be the  flow associated to $f$. Let $p\in \EE$ and
rewrite $\Gamma=\Phi_p(\R)$, $\Omega=\omega_\Phi(p)$. Clearly we
can discard the cases when $\Omega$ is a singleton or a circle. In
particular we suppose $p\in O$. The reader is assumed to be
familiar with the basic facts of the Poincar\'e-Bendixson theory
of sphere flows; regarding this, a good reference is
\cite[Chapter~2]{Aranson}. For instance, $\Phi$ admits no
non-trivial recurrent orbits, that is, $\Gamma\cap
\Omega=\emptyset$.

\begin{lemma}
 \label{teoA-1}
 $\Omega$ is a net (hence a Peano space).
\end{lemma}

\begin{proof}
Due to the compactness of $\EE$, $\Omega$ is a continuum
\cite[Theorem~3.6, p. 24]{BS}. We just need to show that $\Omega$
is  locally connected, because then Theorem~\ref{analytic-4}
easily implies that it is a net. To prove the local connectedness,
according to \cite[Theorem~2, p. 247]{KU}, we must show that if a
continuum $K\subset\Omega$ has empty interior in $\Omega$, then it
is a singleton. Suppose the opposite to find such a continuum $K$
having at least two points. Since $T$ is totally disconnected, $K$
cannot be included in $\Omega \cap T$ and we can find a point $u
\in K \cap O$. Let $X\subset \Omega\cap O$ be a star neighbouring
$u$ and having it as its center (Theorem~\ref{analytic-4}). Two
possibilities arise: either $K$ intersects $X$ exactly at $u$, or
$K\cap X$ contains an arc. Both of them are impossible: the first
one because of the connectedness of $K$, the second one because
$K$ has empty interior in $\Omega$.
\end{proof}

Since we are assuming that $\Omega$ is neither a circle nor a
singleton, it is the union of its non-empty families of edges
(which are countably many) and vertexes. Let $E$ be an edge of
$\Omega$ and $u\in E$. We say that $u$ is \emph{two-sided} if
there is a disk $D$ neighbouring $u$ such that $D$ is decomposed
by $E$ into two components $D_1$ and $D_2$, and $\Gamma$
accumulates at $u$ from both $D_1$ and $D_2$. Otherwise we say
that $u$ is \emph{one-sided}. Recall that if   $u\in E$ is
regular, then there is a flow box $M$ such that $h(0,0)=u$ for the
corresponding homeomorphism $h:[-1,1]\times [-1,1]\rightarrow M$.
Since the arc $h(\{0\}\times [-1,1])$ is transversal to the flow,
it must be intersected monotonically, as time increases, by
$\Gamma$. Hence $u$ is one-sided.

We say that an edge $E$ is \emph{two-sided} if it has some
two-sided point; otherwise, it is called \emph{one-sided}.

\begin{lemma}
 \label{teoA-2}
 Let $E$ be an edge of $\Omega$. Then $E$ is one-sided if and
 only if it is contained in a circle in $\Omega$. Moreover, if $E$
 is two-sided, then all points of $E$ are two-sided.
\end{lemma}

\begin{proof}
The ``if'' part of the first statement is obvious. Next we prove
that if $E$ is not contained in a circle, then it is two-sided.

By Lemma~\ref{teoA-1},  there are disjoint continua
$\Omega_1,\Omega_2$ satisfying $\Omega\setminus E=\Omega_1\cup
\Omega_2$. Use \cite[Theorem~5', p. 513]{KU} to find a circle
$C\subset \EE$ separating $\Omega_1$ and $\Omega_2$. Clearly, we
can assume that $C$ intersects $E$ (hence $\Omega$)  exactly at
one point $u\in O$ which, arguing to a contradiction, we will
suppose one-sided. By Theorem~\ref{analytic-4}, there is a
semi-flow box $M$, with corresponding homeomorphism
$h:[-1,1]\times [0,1]\to M$ and border $B\subset E$, such that
$h(0,0)=u$ and $\Gamma$ accumulates at $B$ from $M$. We can
assume, without loss of generality, that $M$ intersects $C$ at the
arc $L=h(\{0\}\times [0,1])$. After crossing $L$ the semi-orbits
of $\Gamma$ in $M$ enter, as time increases, into one of the open
disks enclosed by $C$, call it $U$. But then $\Gamma$ must also
cross $C$ infinitely many times to escape from $U$, and these
other crossing points cannot belong to $M$ (and hence cannot be
close to $u$ because $u$ is one-sided). Consequently, $\Omega$,
the $\omega$-limit set of $\Gamma$, intersects $C\setminus \{u\}$,
and we get the desired contradiction.

The above argument implies in fact that if $E$ is two-sided, then
all points from $E\cap O$ are two-sided. Since this set is dense
in $E$, all points from $E$ are two-sided.
\end{proof}

Let $\{C_j\}_j$ be the family of circles in $\Omega$. Each $C_j$
decomposes $\EE$ into open disks $R_j$ and $S_j$, which can be
chosen so that the resultant disks $D_j=C_j\cup R_j$ do not
intersect $\Gamma$ (hence $R_j$ is a component of $\EE\setminus
\Omega$ for any $j$). Let $R$ be the component of $\EE\setminus
\Omega$ containing $\Gamma$. Then the family of components of
$\EE\setminus \Omega$ is precisely $\{R\}\cup \{R_j\}_j$. Indeed,
assume that $U$ is a component of $\EE\setminus \Omega$ different
from $R$ and any $R_j$. Lemma~\ref{teoA-2} implies that $\Bd U$
can intersect no edge of $\Omega$; therefore, $\Bd U$ is totally
disconnected and $W=\EE\setminus \Bd U$ is a region. Since $\Bd
W=\Bd U$, $U\subset W$  and both $U$ and $W$ are regions, we get
$U=W$, which is impossible.

Let $A=\Omega\cup \bigcup_j R_j=\EE\setminus R$. We have:

\begin{lemma}
 \label{teoA-3}
 $A$ is a shrub and $\Omega=\Bd A$, the leaves of $A$ being the
 disks $D_j$.
\end{lemma}

\begin{proof}
Since $\Inte\Omega=\emptyset$, we have $\Omega=\Bd A$. Since $R$
is connected, it suffices to show that $A$ is locally connected.

If the family $\{D_j\}_{j=1}^k$ is finite this is simple: just use
the Hahn-Mazurkiewicz theorem to find continuous onto maps
$\varphi:[0,1]\to \Omega$ (here we use Lemma~\ref{teoA-1}),
$\varphi_j:[0,1]\to D_j$, and combine these $k+1$ maps to generate
a continuous map applying $[0,1]$ onto $A$.

If $\{D_j\}_{j=1}^\infty$ is infinite, then the above argument
still works provided that the diameters of the disks $D_j$ tend to
zero. Assume that the opposite is true to find $\delta>0$ and
disks $D_{j_n}$ so that $\diam D_{j_n}=d(u_n,v_n)\geq \delta$ for
appropriate $u_n,v_n\in D_{j_n}$, $n=1,2,\ldots$. We can assume
that the sequence $(u_n)$ converges, say to $u\in A$. If $u$
belongs to one of the open disks $R_j$ or to $\Omega\cap O$, then
we immediately get a contradiction (recall
Theorem~\ref{analytic-4}), so $u$ must belong to $T$. Since $T$ is
totally disconnected, we can find a disk $D$ neighbouring $u$ as
small as needed (in particular, $\diam D<\delta$) so that $\Bd
D\subset O$. If $n$ is large enough, then $u_n\in D$ and
$v_n\notin D$, hence $D_{j_n}$ intersects $\Bd D$. Thus the disks
$D_{j_n}$ accumulate at a point from $O$ and again we get a
contradiction.
\end{proof}

We are ready to finish the proof of Theorem~\ref{teoA}. After
Lemma~\ref{teoA-3}, we are left to show that all odd buds of $A$
are in $T$ and all odd cactuses of $A$ intersect $T$.

Let $P=\Sing(\Phi)\cap O$,  assume that $u\in O$ is an odd bud of
$A$ and let $X$ be an $m$-star as in Theorem~\ref{analytic-4}.
Then $m$ is odd and, since there are no disks $D_j$ near $u$, all
edges ending at $u$ must be two-sided (Lemma~\ref{teoA-2}). In
particular, all points of $X$ must be singular for $\Phi$. Now,
since $P$ is the set of zeros of an analytic function $F:O\to
\mathbb{R}$, it is locally at $u$ a $2n$-star $Y$ for some
non-negative integer $n$ (Theorem~\ref{analytic-2}). Since $X$ is
``odd'' and $Y$ is ``even'', $Y$ strictly includes $X$. This means
that (because all points from $X\setminus \{u\}$ are two-sided)
there is a semi-flow box having two consecutive branches as its
border, and intersecting a branch of $Y$ not included in $X$. This
is impossible, because all singular points of a semi-flow box
belong to its border.

Finally, assume that $K\subset O$ is an odd cactus, when the
$m$-prickly cactus $L$ neighbouring $K$ in $A$ can be assumed to
be included in $O$ as well (hence its sprigs need not be whole
sprigs of $A$). Let $N$ be the set of tips of $L$: again, we
remark that these points may be sprig points when seen in $A$. All
edges ending at $K$ are two-sided, hence all sprigs of $L$ consist
of singular points. Note that there are no singular points outside
$L$ accumulating at $L\setminus N$; otherwise there would be an
arc of singular points in $O$ intersecting $L\setminus N$ at
exactly one point, and we could reason to a contradiction with
similar arguments to those in the paragraph above. The conclusion
is: $G=P\cap L$ is the union of finitely many pairwise disjoint
graphs, which are locally ``even'' at all their vertexes, except
for the $m$ endpoints of $L$. This contradicts the following
general parity property for graphs: if $V=\{v_1, v_2, \ldots,
v_l\}$ is the set of vertexes of a graph $G$ and, for every $i \in
\{1,2, \ldots, l\}$, $r_i$ denotes the order of $v_i$ as a star
point in $G$, then $\sum_{i=1}^{l}{r_i}$ is even (this follows
immediately from the fact that $\sum_{i=1}^{l}{r_i}=2k$, with $k$
being the number of edges of the graph).

\section{Realizing the set of zeros of an analytic function as an
$\omega$-limit set}
 \label{proofint}

This section is entirely devoted to show
Proposition~\ref{realizing} below. If refines some ideas from
\cite{JL} and will be pivotal in the proof of Theorem~\ref{teoB}.
Essentially, it states that the boundary $\Omega$ of a simply
connected region in $\EE$ is the  $\omega$-limit set of a vector
field as smooth as $\Omega$.

\begin{proposition}
 \label{realizing}
Let $O$ be a simply connected region of $\EE$, write $\Omega=\Bd
O$, and let $F:\EE\to \R$ be a $C^\infty$ map, which is analytic
(at least) in $O$, and satisfies $F(u)\neq 0$ for any $u\in O$ and
$F(u)=0$ for any $u\in \Omega$. Then there is a $C^\infty$-vector
field $f$ in $\EE$, which is analytic wherever $F$ is (in
particular, in $O$), and such that its associated flow has
$\Omega$ as one of its $\omega$-limit sets.
\end{proposition}

In what follows we assume, without loss of generality and after
applying appropriate analytic transformations, that the north pole
$p_N=(0,0,1)$ of $\EE$ belongs to $\Omega$, that the south pole
$p_S=(0,0,-1)$ belongs to $O$, and that the meridian $I_0$
consisting of the points $(\sqrt{1-z^2},0,z)$, $z\in [-1,1]$, is
included in $O\cup \{p_N\}$.  (More in general, by a
\emph{meridian} we mean an arc in $\EE$ having $p_N$ and $p_S$ as
its endpoints and which is included in $O\cup \{p_N\}$.)

As it turns out, the vector field $f$ we are looking for can be
explicitly derived from $F$, which immediately guarantees that it
satisfies the smoothness requirements from the theorem. Namely,
let $\|\cdot\|$ denote the euclidean norm, let $G:\R^3\setminus
\{(0,0,0)\}\to \R$ be given by $G(u)=F^2(u/\|u\|)$, and define
$f:\R^3\setminus \{(0,0,0)\}\to \R^3$, $f=(f_1,f_2,f_3)$, as
follows:
 \begin{eqnarray*}
 f_1(x,y,z) &=& 2 z (y-x) G(x,y,z) + (x^2 + y^2)
  \left( -y \frac{\partial G}{\partial z}(x,y,z)
  + z \frac{\partial G}{\partial y}(x,y,z) \right), \\
 f_2(x,y,z) &=& -2 z (x + y) G(x,y,z) + (x^2 + y^2)
  \left( x \frac{\partial G}{\partial z}(x,y,z)
  - z \frac{\partial G}{\partial x}(x,y,z) \right), \\
 f_3(x,y,z) &=& (x^2 + y^2)
  \left(2 G(x,y,z) + y \frac{\partial G}{\partial x}(x,y,z)
  -  x \frac{\partial G}{\partial y}(x,y,z)\right).
\end{eqnarray*}
It is easy to check that $f(u)\cdot u=0$ for any $u$. Hence $f$,
when restricted to $\EE$, induces a vector field on $\EE$. Observe
that, because of the definition of $G$, all points of $\Omega$ are
singular points for the corresponding equation
 \begin{equation}
  \label{sistema}
  u'=f|_\EE(u).
 \end{equation}
On the other hand, although $G$ is positive on $O$, there may be
many singular points of \eqref{sistema} in $O$ ($p_S$, for
instance, is one of them).

Next we will show, through a sequence of lemmas, that $\Omega$ is
an $\omega$-limit set for \eqref{sistema}, but first the point
behind the definition of $f$ must be clarified. For this, consider
the semi-space $U=\{(x,y,z)\in \R^3: z<1\}$ and the map
$\pi:U\to\R^2$ given by $\pi(x,y,z)=(x/(1-z), y/(1-z))$, which is
of course the stereographic projection when restricted to $\EE$.
Recall that if a meridian $I$ is given, then there exists an
analytic map $\Lambda_I:\R^2\setminus \pi(I\setminus\{p_N\})\to
\R$ such that $\Lambda_I(x,y)\in \arg(x+iy)$ for any $(x,y)$.
Likewise, let $U_I=U\setminus \pi^{-1}(\pi(I\setminus\{p_N\}))$
and define $\Theta_I:U_I\to \R$ by $\Theta_I=\Lambda_I\circ \pi$.
Note that $\Theta_I$ can be locally written as
$\Theta_I(x,y,z)=k\pi+\arctan(y/x)$ or
$\Theta_I(x,y,z)=k\pi+\arccot(x/y)$ for some integer $k$, and then
 $$
 \nabla \Theta_I(x,y,z) =
 \left(\frac{\partial \Theta_I}{\partial x}(x,y,z),
 \frac{\partial \Theta_I}{\partial y}(x,y,z),
 \frac{\partial \Theta_I}{\partial z}(x,y,z)\right)
 =\left(\frac{-y}{x^2+y^2},\frac{x}{x^2+y^2},0\right).
 $$
Finally, write $\rho(x,y,z)=(x^2+y^2)G(x,y,z)$,
$J_I(u)=\rho(u)e^{-2\Theta_I(u)}$ and $H_I(u)=\log J_I(u)$. While
$J_I$ is well defined in $U_I$, $H_I$ only makes sense in the open
set $V_I=U_I\cap G^{-1}((0,\infty))$. Still, observe that
$O\setminus I\subset V_I$.

Fix a meridian $I$. The key property of $f$ is that, as it can be
easily checked, we can write it as
 $$
 \text{$f(u)=\rho(u) (\nabla H_I(u)\times u)$ whenever $u\in V_I$.}
 $$
This has the important consequence that $\nabla H_I(u)\cdot u'=0$
for some relevant (connected) smooth curves
$u(t)=(x(t),y(t),z(t))$ in $O\setminus I$, which means that $H_I$
(and consequently $J_I$) is constant on them. Such is the case,
for instance, if $u(t)$ is a solution of the system
\eqref{sistema}, because then
 $$
 \nabla H_I(u)\cdot u'=
 \rho(u) \nabla H_I(u)\cdot (\nabla H_I(u)\times u)=0,
 $$
and also if all points of the curve $u(t)$ are singular, because
then $\nabla H_I(u)\times u=0$, which implies that $\nabla
H_I(u)=\kappa(u)u$ for some  scalar map $\kappa$, and therefore
 $$
 \nabla H_I(u)\cdot u'=
 \kappa(u)(u\cdot u')=0,
 $$
the last equality just following from the fact that $u(t)$ is a
curve in the sphere $\EE$.

The above properties can be exploited further. Firstly,
Theorem~\ref{analytic-2} implies that if $\Phi$ is the flow
associated to \eqref{sistema}, then $J_I$ is in fact locally
constant in $\Sing(\Phi)\cap (O\setminus I)$, hence constant on
each of the components of this set. On the other hand, if $p\in
O\setminus \{p_S\}$, then we cannot automatically guarantee the
constancy of some concrete map $J_I$ on the whole maximal solution
$\Phi_p(t)=(x_p(t),y_p(t),z_p(t))$, because although the orbit
lies in $O$ it needs not be fully included in any region
$O\setminus I$. Still, it is clearly possible to find a continuous
choice of the angle (that is, a continuous map $\theta_p:\R\to \R$
satisfying $\theta_p(t)\in \arg(x_p(t)+iy_p(t))$ for any $t\in
\R$), so that
 $$
 w_p(t)=\rho(\Phi_p(t)) e^{-2 \theta_p(t)}
 $$
is constant. We gather these results as a lemma:

\begin{lemma}
 \label{realizing-1}
 The following statement holds:
  \begin{itemize}
  \item[(i)] If $I$ is a meridian, then $J_I$ is constant on any
  component of $\Sing(\Phi)\cap (O\setminus I)$.
  \item[(ii)] If $p\in O\setminus \{p_S\}$ and $I$ is a meridian,
  then  $J_I$ is constant on every semi-orbit of
  $\Phi_p(\R)$ included in $O\setminus I$; moveover,
  the above map $w_p(t)$ is constant.
  \end{itemize}
\end{lemma}

\begin{lemma}
 \label{realizing-2}
 The south pole $p_S$ is a repelling focus for $\Phi$.
\end{lemma}

\begin{proof}
It suffices to show that the same statement is true, with respect
to the origin, when we transport the system \eqref{sistema} to
$\R^2$ via the local chart $(x,y)\mapsto (x,y,\sqrt{1-x^2-y^2})$.
The obtained system is
 $$
 g(x,y)=(f_1(x,y,\sqrt{1-x^2-y^2}), f_2(x,y,\sqrt{1-x^2-y^2})).
 $$
Now a direct calculation shows that the jacobian matriz of $g$ at
$(0,0)$ is
 $$
 Jg(0,0)=
 \begin{bmatrix}
 2G(p_S) & -2G(p_S) \\
 2G(p_S) &  2G(p_S)
 \end{bmatrix},
 $$
 its eigenvalues being $2G(p_S)(1\pm i)$. Since $G(p_S)>0$, the
 lemma follows.
\end{proof}

\begin{lemma}
 \label{realizing-3}
Let $P$ be the set of singular points in $O$ which are non-trivial
$\omega$-limit sets (that is, $p\in P$ if and only if there is
$q\neq p$ such that $\omega_\Phi(q)=\{p\}$). Then, for any $p\in
P$, there are only finitely many orbits having $p$ as its
$\omega$-limit set. Moreover, $P$ is discrete, that is, all its
points are isolated, hence countable.
\end{lemma}

\begin{proof}
Let $p\in P$, fix a meridian $I$ not containing $p$ and say
$J_I(p)=a$. Find a small star $X\subset O\setminus I$ neighbouring
$p$ in $\Sing\Phi$ (recall that $X$ becomes to a $0$-star, that
is, just the point $p$, when $p$ is isolated in $\Sing\Phi$). Now
realize that, by Lemma~\ref{realizing-1} and continuity, $J_I$
also equals $a$ on $X$ and all small semi-orbits ending at points
from $X$. Since $J_I^{-1}(\{a\})$ is analytic, and $J_I$ cannot be
constant on $O\setminus I$ (because then $f$ would vanish on the
whole $O$, which is not true in view of Lemma~\ref{realizing-2}),
the lemma follows immediately from Theorem~\ref{analytic-2}.
\end{proof}

\begin{lemma}
 \label{realizing-4}
Let $p\in O$ and assume that $\alpha_\Phi(p)=p_S$  and
$\omega_\Phi(p)$ is not a singular point of $O$. Then
$\omega_\Phi(p)=\Omega$.
\end{lemma}

\begin{proof}
Rewrite $u(t)=\Phi_p(t)$, $\Gamma=\Phi_p(\R)$,
$\theta(t)=\theta_p(t)$, $w(t)=w_p(t)$, $\Omega'=\omega_\Phi(p)$.

First we show that $\Omega'\subset \Omega$. Suppose not to find
$q\in \Omega'\cap O$ and (using Theorem~\ref{analytic-4}) a
semi-flow box $M\subset O$ whose border $B$ is in $\Omega'$ and such
that $\Gamma$ accumulates at $B$ from $M\setminus B$. If
$h:[-1,1]\times[0,1]\to M$ is the corresponding homeomorphism,
then the points $q_n=u(t_n)$ ($t_n\geq 0$)  of $\Gamma$
intersecting $A=h(\{0\}\times [0,1])$ converge monotonically, as
time increases, to $q$. If $A_n$ is the arc in $A$ with endpoints
$q_n$ and $q_{n+1}$, then two possibilities arise: either one of
the circles $C_n=A_n\cup u([t_n,t_{n+1}])$ separates $q$ and
$p_N$, or neither of them does.

Assume that the first possibility holds. If, say, $C_{n_0}$
separates $q$ and $p_N$, then there is a meridian $I$ such that
neither $u([t_{n_0},\infty))$ nor $\Omega'$ intersect it (we are
using $\alpha_\Phi(p)=p_S$).  Hence $J_I$ equals to a constant $a$
on $u([t_{n_0},\infty))$ (Lemma~\ref{realizing-1}(ii)). Moreover,
by continuity, $J_I= a$ on $\Omega'$ as well. Now, since
$J_I^{-1}(\{a\})$ is an analytic set which it is not locally a
star at $q$, we get $J_I^{-1}(\{a\})=O\setminus I$
(Theorem~\ref{analytic-2}), which is impossible.

If the second possibility holds, then all curves $C_n$ have the
same winding number $\nu\in \{-1,1\}$ around $p_N$, and
$\theta(t_{n+1})-\theta(t_n)\to 2\pi \nu$ as $n\to \infty$.
Therefore, $|\theta(t_n)|\to \infty$. Since $\rho$  is positive in
$q$, it is impossible that $w(t)$ is constant, contradicting
Lemma~\ref{realizing-1}(ii). This concludes the proof that
$\Omega'\subset \Omega$.

We are now ready to prove $\Omega'=\Omega$. Note firstly that,
since $d(u(t),\Omega)\to 0$ as $t\to \infty$ (because
$\Omega'\subset \Omega$) and $G$ vanishes at $\Omega$, the only
way for Lemma~\ref{realizing-1}(ii) to hold is that $\theta(t)\to
-\infty$ as $t\to \infty$. We can, of course, assume
$\theta(0)\geq 0$, hence the last and first numbers $t_n$ and
$s_n$ respectively satisfying $\theta(t_n)=-2\pi (n-1)$ and
$\theta(s_n)=-2\pi n$, $n\geq 1$, are well defined. Moreover, if
$A_n$ are the arcs in $I_0$ with endpoints $p_n=u(t_n)$ and
$q_n=u(s_n)$, then all circles $C_n=A_n\cup u([t_n,s_n])$ have
winding number $-1$ around $p_N$, hence they separate $p_S$ and
$p_N$. Let $R_n$ denote the open disk in $O$ (recall that $O$ is
simply connected) enclosed by $C_n$ and construct a sequence of
disks $(M_k)_{k=1}^\infty$ in $O$ such that $p_S\in M_k$ for any
$k$ and $\bigcup_{k=1}^\infty M_k= O$. If $M_k$ is given, say
$d(M_k,\Omega)=\delta>0$, and $n$ is large enough such that
$\max_{c\in C_n} d(c,\Omega)<\delta$, then we get $C_n\cap \Bd
M_k=\emptyset$, which together with $p_S\in M_k\cap R_n$ implies
$M_k\subset R_n$. Therefore, we get $O=\bigcup_{n=1}^\infty R_n$,
and recall that $\Bd O=\Omega$. This implies that if $q\in \Omega$
and $W$ is an arbitrarily small neighbourhood of $q$, then there
is some $R_n$ (and therefore some $C_n$) intersecting $W$. Add to
this that $\diam A_n\to 0$ to easily conclude $\Omega'=\Omega$, as
we desired to prove.
\end{proof}

Proposition~\ref{realizing} easily follows from the previous
lemmas. Namely, there are at most countably many non-trivial
orbits in $O$ whose $\omega$-limit set is a singular point of $O$
(Lemma~\ref{realizing-3}). In particular, there is $p\in O$ such
that $\alpha_\Phi(p)=\{p_S\}$  and $\omega_\Phi(p)$ is not a
singular point of $O$ (Lemma~~\ref{realizing-2}). Then
$\omega_\Phi(p)=\Omega$ by Lemma~~\ref{realizing-4}.

\section{Proof of Theorem~\ref{teoB}}
 \label{proofB}

The main idea behind the proof of Theorem~\ref{teoB} is simple
enough: if $\Omega$ is the boundary of a shrub $A$, then we first
find a homeomorphism $h:\EE\to \EE$ mapping $\Omega$ to a set
$h(\Omega)$, ``as analytic as possible'', and then apply
Proposition~\ref{realizing} to realize it as the $\omega$-limit
set of an ``as analytic as possible'' vector field. The procedure
is easier when the leaves of $A$ are ``well-behaved'': therefore,
we discuss this case in advance, in Subsection~\ref{casosimple}
below. Then we deal with the general case in
Subsection~\ref{general-case}.

\subsection{The simple case}
 \label{casosimple}

The ``as analytic as possible'' set we have just referred to will
be constructed by, informally speaking, pasting segments and
hypocycloids. Thus, to begin with, we must guarantee the
analyticity of these sets.

\begin{lemma}
 \label{analytic-segment}
 The open segment $(-1,1)\times \{0\}$ is analytic in $\R_\infty^2$ minus
 its two endpoints.
\end{lemma}

\begin{proof}
Let $F: \R^3 \to \R$ be given by
$$
F(x,y,z)=y^2 +(\sqrt{z^2 + y^2} + z)^2.
$$
It is well-defined and continuous in $\R^3$ and analytic in $\R^3
\setminus \{(x,0,0): x \in \R\}$. The restriction of $F$ to $\EE$
is a continuous map whose zeros are the points $(x,y,z) \in \EE$
with $y=0$ and $z \leq 0$, and which is analytic in $\EE \setminus
\{(-1,0,0), (1,0,0)\}$. Call $Z_F$ the set of zeros of $F$. The
image of the set $Z_F$ under the stereographic projection $\pi_N$,
which is exactly the set $(-1,1)\times \{0\}$, is then analytic in
$\R_\infty^2$  minus  its two endpoints.
\end{proof}

For every positive integer $k \geq 3$, \textit{the $k$-cusped
hypocycloid} $H_k$ is the plane curve defined by the parametric
equations
 \begin{equation}
 \label{parhyp}
 \begin{cases}
  x_k(\theta)=
   (k -1) \cos{\theta} + \cos((k-1) \theta), \\
  y_k(\theta)=
   (k -1) \sin{\theta} - \sin((k-1) \theta),
 \end{cases}
 \end{equation}
$\theta \in \R$. Observe that the only values of the parameter
$\theta \in [0, 2 \pi)$ where the derivative of $(x_k(\theta),
y_k(\theta))$ vanishes are $2\pi j/k$, $0\leq j\leq k-1$, thus
arising the \emph{cusps} $k e^{2\pi j\mathbf{i}/k}$ of $H_k$. When
$k=2r$ is even, these cups can be seen as $r$ pairs of opposed
points: they are symmetric  with respect to the origin $\0=(0,0)$,
and a straight line passing through a cusp and the origin meets
the hypocycloid exactly at the cusp and the opposed one: see
Figure~\ref{figura2}.

\begin{figure}
\begin{center}
\begin{tikzpicture}[scale=0.2]
 \draw[smooth,samples=100,domain=0:360]
 plot(
 {32/5*cos(\x)+8/5*cos(4*\x)-30},
 {32/5*sin(\x)-8/5*sin(4*\x)}
 );
 \draw[smooth,samples=100,domain=0:360]
 plot(
 {7*cos(\x)+cos(7*\x)},
 {7*sin(\x)-sin(7*\x)}
 );
 \draw (-10,0.)-- (10,0.);
 \draw (7.353910524340094,7.353910524340094)--
  (-7.353910524340094,-7.353910524340095);
 \draw [fill=black] (0.,0.) circle (3pt);
\end{tikzpicture}
\end{center}
\caption{\label{figura2} The $5$-cusped hypocycloid (left) and the
$8$-cusped hypocycloid with some arcs (right).}
\end{figure}
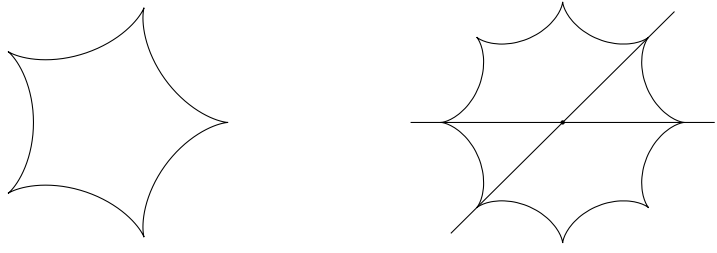

\begin{lemma}
 \label{analytic-hypo}
For every $k \geq 3$, the $k$-cusped hypocycloid  $H_k$ is an
algebraic set (the set of zeros of a polynomial map), and then an
analytic set in $\RR$.
\end{lemma}

\begin{proof}
After writing  $z=e^{\mathbf{i} \theta}$ and $n=k-1$, the
hypocycloid can be described, in parametric form, as
\begin{eqnarray*}
 X(z) &=& n \frac{z + z^{-1}}{2} + \frac{z^{n} + z^{-n}}{2}, \\
 Y(z) &=& n \frac{z - z^{-1}}{2 \mathbf{i}} - \frac{z^{n} - z^{-n}}{2 \mathbf{i}},
\end{eqnarray*}
where now $z \in \Sd\subset \C$. The zeros of the polynomial
$F(x,y)$ we are looking for must be, exactly, those points $(x,y)
\in \R^2$ for which there exists some $z \in \Sd$ such that
$X(z)=x$ and $Y(z)=y$ or, equivalently, the points $(x,y)$ for
which the complex polynomials
\begin{eqnarray*}
p_x(z)&:=& z^{2 n} + n z^{n+1} - 2 x z^n   + n z^{n-1} + 1, \\
q_y(z)&:=& z^{2 n }- n z^{n+1} + 2 \mathbf{i} y z^n + n z^{n-1} -
1,
\end{eqnarray*}
have a common root in $\Sd$, which amounts to simply say that
$p_x(z)$ and $q_y(z)$ have some common root because, as we next
show, if $z_0$ is a root of both $p_x(z)$ and $q_y(z)$, then
$z_0\in \Sd$.

Write $w=x + \mathbf{i} y$, $u(z)= z^{n+1} - \overline{w} z + n$
and $v(z)= n z^{n+1} - w n z^{n} + 1$. Then $p_x(z) + q_y(z) = 2
z^{n-1} u(z)$ and $p_x(z) - q_y(z) = 2 v(z)$,  hence $z_0$ is also
a common root of $u(z)$ and $v(z)$ (because $z_0\neq 0$). Now
notice that, for every $z \in \C \setminus \{0\}$,
$\overline{z}^{n+1} v(1/\overline{z}) = \overline{u(z)}$. This
means that there is another root $z_1$ of $u(z)$ satisfying $z_0=
1/\overline{z_1}=z_1/\left|z_1\right|^2$. If we write $\sigma=z_1$
and $r=1/\left|z_1\right|$, we get $u(\sigma)=u(r^2 \sigma)=0$,
that is,
\begin{eqnarray*}
\sigma^{n+1} - \overline{w} \sigma + n &=& 0,\\
r^{2n + 2} \sigma^{n+1} - r^2 \overline{w} \sigma + n &=& 0,
\end{eqnarray*}
which also implies $(r^{2 n} - 1)r^2\sigma^{n+1} + (1 - r^2) n=0$.
Taking moduli, this last expression can only be true if
 $$
 (r^2 -1)n r^{n-1} = r^{2n}- 1=
 (r^2 -1)(1 + r^2 + \cdots + r^{2(n-1)}),
 $$
that is, either $r=1$ or
 \begin{eqnarray*}
 0&=&1 + r^2 + \cdots + r^{2(n-1)}-n r^{n-1}\\
  &=&\sum_{j=0}^{n-1} r^{2j}-r^{n-1}\\
  &=&\sum_{j=0}^{n'-1} r^{2j}-2r^{n-1}+r^{2(n-1-j)}\\
  &=&\sum_{j=0}^{n'-1}r^{2j}(1-2r^{n-1-2j}+r^{2(n-1-2j)})\\
  &=&\sum_{j=0}^{n'-1}r^{2j}(1-r^{n-1-2j})^2,
 \end{eqnarray*}
where $n'$ is the integer part of $n/2$. This is only possible if,
again,  $r=|z_1|=1$. Then $|z_0|=1$ as well, as we desired to
show.

We have proved that $(x,y)$ belongs to the hypocycloid if and only
if the polynomials $p_x(z)$ and $q_y(z)$ have a common root, which
is also equivalent to $R(p_x,q_y)=0$, with $R(p_x,q_y)$ denoting
the resultant of the polynomials $p_x(z)$ and $q_y(z)$
\cite[Theorem~8.27, p. 151]{SPINDLER}. Now, since $R(p_x,q_y)$ is
a determinant calculated on the coefficients of $p_x(z)$ and
$q_y(z)$, we get $R(p_x,q_y)=P(x,y)$, where $P(x,y)$ is polynomial
(with complex coefficients) on $x$ and $y$. If  $P(x,y)=P_1(x,y) +
\mathbf{i} P_2(x,y)$, with $P_1(x,y)$ and $P_2(x,y)$ polynomials
with real coefficients, we finally obtain that the hypocycloid is
the set of zeros of $F(x,y)=P_1(x,y)^2 + P_2(x,y)^2$.
\end{proof}

\begin{remark}
 \label{algebraic}
An algebraic set $A\subset \RR$  \emph{needs not be} analytic in
$\R_\infty^2$ (equivalently, analytic in $\EE$ via the
stereographic projection): it is $\pi_N^{-1}(A)\cup \{p_N\}$ which
is analytic in $\EE$. In fact, if $A$ is the set of zeros of a
polynomial $P(x,y)$, then $\pi_N^{-1}(A)$ is the set of zeros in
$\EE$ of the map $Q(x,y,z)=P(x/(1-z),y/(1-z))$, hence
$\pi_N^{-1}(A)\cup \{p_N\}$ is the set of zeros in $\EE$ of
$(1-z)^n Q(x,y,z)$ which, if $n$ is large enough, is well-defined
and polynomial in the whole $\R^3$.

Circumferences, on the other hand, are analytic in $\R_\infty^2$
because they can be seen, in $\EE$, as the intersection of a plane
with the sphere, that is, the restriction to $\EE$ of the set of
zeros of an affine map in $\R^3$.
\end{remark}

In what follows, by a \emph{hypocycloid} we mean, in fact, any
affine deformation $H$ in $\RR$ of  (the disk enclosed by) some
$2r$-cusped hypocycloid as previously defined, while, as usual, a
\emph{segment} is an affine deformation in $\RR$ of the arc
$[-1,1]\times \{0\}$. The non-smooth points at the boundary of $H$
are still called its \emph{cusps}, and two cusps of $H$ are
\emph{opposed} if there is a segment in the hypocycloid connecting
them. Such a segment is called a \emph{diameter} of $H$, and all
diameters intersect at a unique point, the \emph{center} of $H$.
The \emph{radii} of $H$ are then the segments connecting the
center of $H$ with its cusps. From Lemmas~\ref{analytic-segment}
and \ref{analytic-hypo} it follows immediately that any segment is
analytic in $\R_\infty^2$ minus its endpoints, and the boundary of
any hypocycloid is analytic in $\RR$.

Let $M_1,M_2$ be two segments  intersecting exactly at a common
endpoint $p$. We say that $M_1$ and $M_2$ are \emph{aligned} if
$M_1\cup M_2$ is a segment. We say that $M_2$ is \emph{positively}
(respectively, \emph{negatively}) \emph{biased with respect to
$M_1$} if $M_1$ and $M_2$ are not aligned and, after fixing a
segment $M_0$ intersecting $M_1$ and $M_2$ exactly at $p$, and
such that $M_1$ and $M_0$ are aligned, the triod $(M_1,M_0,M_2)$
(when seen in $\EE$ via the identification $\R_\infty^2 \cong\EE$)
is positive (respectively, negative). A segment $M$ and a
hypocycloid $H$ are \emph{aligned} if they intersect exactly at a
cusp of $H$ and $M$ and the radius of $H$ ending at this cusp are
aligned. Finally, two hypocycloids are \emph{aligned} it they
intersect at a common cusp and the corresponding radii are
aligned.

We intend to construct our ``almost analytic'' shrubs via an
appropriate skeleton structure so that we can benefit from
Propositions~\ref{skeleton1} and \ref{skeleton2}. For this we will
need some special thick arcs, not just consisting of segments and
hypocycloids, but additionally ``twisted'' (biased) according to
some rules. The following lemma explains how to do it.

\begin{lemma}
 \label{hypo-segment}
Let $A$ be a thick arc in $\RR$ and fix a proper pair of endpoints
$w,z$ of $A$. Let $Q$ be a set of disconnecting points of $A$ and
let $\mathcal{D}$ be a subfamily of leaves of $A$. We assume that
the union set of $Q$ and the centers of the leaves from
$\mathcal{D}$ is discrete, and that no leaf from $\mathcal{D}$
contains $z$. Assign to each $p\in Q$ (respectively, to each leaf
$D\in\mathcal{D}$) a number $\phi(p)=\pm 1$ (respectively,
$\phi(D)=\pm 1$). Then there are a thick arc $A'$ and a
homeomorphism $h:\RR\to\RR$ with $h(A)=A'$ (homotopic to the
identity when extended to $\R_\infty^2$) satisfying the following
properties:
 \begin{itemize}
 \item[(i)] Each leaf of $A'$ is a hypocycloid and all non-cusp
points of the leaves of $A'$ are exterior points of $A'$. Also,
$h$ maps $w$ to the origin $\0$ and, if $z'=h(z)$ and $\0$ and/or
$z'$ belong to some leaf of $A'$, then they are cusp points. If,
moreover, $I'$ denotes the arc in $A'$ connecting $\0$ and $z'$
and intersecting each leaf at exactly two radii, then any subarc
of $I'$ neither intersecting $h(Q)$ nor containing the center of
any hypocycloid $h(D)$, $D\in \mathcal{D}$, is a segment.
 \item[(ii)]
$A'\subset C_\rho:=\{\0\}\cup \{(x,y): x>0, |y|<\rho x\}$ for some
$\rho>0$.
 \item[(iii)]  Let $p\in Q$ or $ D\in \mathcal{D}$ and let
$p'$ be, depending on the case, either $h(p)$ or the center of the
hypocycloid $h(D)$. Let $M_1'$ and $M_2'$ be the maximal segments
in $I'$ having $p'$ as its common endpoint, with $M_1'$ being the
closer segment to $\0$. If, according to the case, $\phi(p)=1$ or
$\phi(D)=1$ (respectively, $\phi(p)=-1$ or $\phi(D)=-1$), then
$M_2'$ is positively (respectively, negatively) biased with
respect to $M_1'$.
\end{itemize}
\end{lemma}

\begin{proof}
First of all, we map $A$ via a homotopic to the identity
homeomorphism $f$ to a thick arc $A_0$ which is the union of the
segment $I_0=[0,1]\times \{0\}$ (with $w$ mapped to $\0$) and some
balls with diameters included in this segment. If, additionally,
there is a sequence of leaves $(B_n)_{n=1}^\infty$ monotonically
accumulating at $\0$, we can assume that $\diam(B_n)=1/2^{n+1}$
for any $n$ and that any leaf between $\0$ and $B_n$ has diameter
less than $1/2^{n+1}$, which ensures that $A_0\subset C_{1/2}$.

Observe that both $\mathcal{D}$ and $Q$ are countable. Our
construction will proceed in two steps. Firstly we assume that $Q$
is empty, then we consider the general case.

Say $\mathcal{D}=\{D_i\}_{i=1}^\infty$ (if $\mathcal{D}$ is
finite, then the argument is analogous but simpler), and let $p_i$
denote the center of the ball $f(D_i)$. We arrive to $A'$ via a
sequence of intermediate homeomorphisms $h_i:\RR\to \RR$ whose
compositions $h_i^*=h_i\circ\ldots\circ h_1$ will converge to a
homeomorphism $h^*$. This homeomorphism almost provides the set
$A'$ we are looking for, except that the leaves of $A^*=h^*(A_0)$
are not hypocycloids but balls. To conclude one just have to
replace $f$ by another homeomorphism $f^*$ mapping the leaves of
$A$ to slightly deformed hypocycloids having their cusps in the
boundaries of the balls of $A_0$ (so that these
``pseudo-hypocycloid'', after applying $h^*$, become real
hypocycloids). Then $h=h^*\circ f^*$ does the job.

All homeomorphisms $h_i$ are constructed (modulo a translation and
a rotation) using a map $g_\delta$, with $\delta=\delta_i$ a
rational number small enough (in absolute value), which is defined
as follows. Let $\tau_\delta:[-1/2,1/2]\to [-1/2,1/2]$ the (five
pieces) piecewise affine homeomorphism mapping $[-1/4+2|\delta|,
1/4-2|\delta|]$ onto $[-1/4+2|\delta|+\delta,
1/4-2|\delta|+\delta]$ and leaving invariant the intervals
$[-1/2,-1/4]$ and $[1/4,1/2]$. Then, using polar notation, we
define $g_\delta(r e^{2\pi \textbf{i}\theta})=r e^{ 2\pi
\textbf{i} \tau_\delta(\theta)}$. Thus $g_\delta$ leaves invariant
the second and the third quadrant and slightly rotates, with angle
$2\pi\delta$, ``most'' of points in the first and fourth quadrant.

Now, starting from $I_0$ and  $A_0$, and via the procedure
$h_i(I_{i-1})=I_i$ and $h_i(A_{i-1})=A_i$, we get some polygonals
$I_i$ and thick arcs $A_i$ so that the angle points of $I_i$ are
exactly $h_i^*(\{p_1,\ldots,p_i\})$, and if $D$ is a leaf of $A_0$
of center $q$, then the corresponding leaf $h_i^*(D)$ is a ball of
center $h_i^*(q)$ and the same radius as $D$. More precisely, to
define $h_i$, we first transport the point $h_{i-1}^*(p_i)$ (we
mean $h_0^*=\Id$) to $\0$, and the segment of $I_{i-1}$ containing
it to the $x$-axis (so that its closest part to $\0$ falls to the
left of $\0$), via some appropriate translation and rotation. Then
we compose with a certain $g_{\delta_i}$, with the small rational
number $\delta_i$ being positive or negative according to the sign
of $\phi(p_i)$ (and ensuring that all transported points from
$A_{i-1}$, except some of the ball containing $h_{i-1}^*(p_i)$,
are either $2\pi\delta_i$-rotated or left invariant), and finally
apply the reversed rotation and translation. It is clear that if
the numbers $\delta_i$ are sufficiently small, then the maps
$h_i^*$, and similarly their inverses $(h_i^*)^{-1}$, converge
uniformly to a homeomorphism $h^*$ and its inverse $(h^*)^{-1}$.

Observe that, in order this geometrical construction to work
properly, we must be sure that the perpendicular line to $I_{i-1}$
passing through $h_{i-1}^*(p_i)$ do not intersect $A_{i-1}$ except
at the ball enclosing the point. If fact, after fixing a sequence
of numbers $0<\mu_i<1$, say $\mu_i=1-1/2^i$, close enough to $1$
so that $\prod_{i=1}^\infty \mu_i>0$, more will be true: if $u\in
A_{i-1}$ disconnects $A_{i-1}$, and $\sigma$ is the perpendicular
to $I_i$ passing through $u$, then not only $\sigma$ will
intersect $A_{i-1}$ just at $u$, and similarly the perpendicular
$\sigma'$ to $I_i$ passing through $u'=h_i(u)$  will intersect
$A_i$ just at $u'$, but for any $v\in A_{i-1}$, $v'=h_i(v)$, we
have $d(v',\sigma')\geq \mu_i d(v,\sigma)$. Clearly, this can be
guaranteed, inductively, just using numbers $\delta_i$ small
enough. Then the limit polygonal $I^*=h^*(I_0)$ satisfies the
analogous property for $A^*$:  if $v^*\in A^*$ disconnects $A^*$,
and $\sigma^*$ is the perpendicular to $I^*$ passing through
$v^*$, then $\sigma^*$ intersects $A^*$ just at $v^*$. (This will
be needed in the proof of the general case.) Also, observe that
$A^*$ is included in the right half-plane (and intersects the
$y$-axis just at $\0$). If we are in the accumulating case
described in the first paragraph, the construction can be refined
to guarantee $A^*\subset C_{1/2}$. After replacing the balls in
$A^*$ by hypocycloids with sufficiently many cusps (this can be
done without losing connectedness: this is the reason why we
required the numbers $\delta_i$ to be rational), we get, as
explained before, the homeomorphism $h$ and the thick arc $A'$ we
are looking for. Just one detail is pending: guaranteeing that
$A'$ is included in some set $C_\rho$. If there is an arc (a
segment) in $\Bd A'$ or a leaf (a hypocycloid) in $A'$ containing
$\0$, this is trivial; otherwise we are in the accumulating case,
and $C_{1/2}$ can be used. We remark that $I^*$ is, indeed, the
arc $I'$ from the condition (i) in the lemma.

Recall that we have assumed, until now, that $Q$ is empty. If
$Q\neq \emptyset$, then the thick arc we obtain applying the
previous construction to $\mathcal{D}$ is not yet that we need.
Therefore, instead of calling that set $A'$, we call it $B_0$ and
use it, together with the limit polygonal $J_0=I^*$, as the new
starting point for a similar construction, again based on the maps
$g_\delta$, now around the points from $Q$. Such a construction,
providing the desired set $A'$, is possible due to the
above-mentioned property of $A^*$, and hence of $B_0$, of
separation by perpendicular lines to $J_0$.
\end{proof}

\begin{remark}
 \label{rem-hypo-segment}
Each hypocycloid from the previous set $A'$ is rounded, that is,
all their cusps belong to the same circumference. After composing
with the linear map $(x,y)\mapsto (x,\epsilon y/\rho)$, and at the
cost of losing roundness (but still having hypocycloids), we can
carry $A'$ into $C_\epsilon$ for any $\epsilon$ as small as we
wish.

Also, observe that the number of cusps of each hypocycloid
$D'=h(D)$ of $A'$ can be chosen sufficiently large so that, if the
number $n=n(D)$ is fixed in advance, and $P'$ is the four (or two)
points set containing the intersection points of $I'$ with $\Bd
D'$ and their opposed cusps in $D'$, then the number of cusps of
$\Bd D'$ between the points of $P'$ is at least $n$.
\end{remark}

It is already time to define the special type of shrub we are
concerned of in this subsection. We say that a shrub $A$ is
\emph{simple} if all nodes belonging to leaves of $A$ are star
points of $\Bd A$. In particular, all cactuses and prickly
cactuses are simple shrubs. A shrub is \emph{very simple} if it is
simple and has no odd cactuses. If $A$ is very simple, then we
denote by $P_A$ the set of nodes of $A$ which are not odd buds of
$A$ (that is, all nodes belonging to some leaf or having even
order) and by $\mathcal{E}_A$ the family of pieces (leaves and
sprigs) of $A$. These sets are, to say so, the places where there
is the ``peril'', to be avoided, of losing analyticity in the
ensuing construction.

Let $A$ be a very simple shrub. If $p\in P_A$, then it disconnects
$A$ into $k$ components, whose closures $B_1,\ldots,B_k$ are
simple shrubs as well, although maybe not very simple. Denote by
$E_1,\ldots,E_k$ the corresponding pieces $E_j\subset B_j$
containing $p$. We divide these pieces into three classes. A piece
$E_j$ is \emph{rigid} \emph{(with respect to $p$)} it either $E_j$
is a sprig or $B_j$ is not very simple (note that, in this second
case, $B_j$ has exactly one odd cactus, and $E_j$ is a leaf from
this cactus). We say that $E_j$ is \emph{flexible} if there is an
infinite connected union of leaves in $B_j$, one of them being
$E_j$. If $E_j$ is neither rigid nor flexible, then is it called
\emph{bland}. Let us assume that a subfamily of non-bland pieces
$\{E_{j_1},\ldots,E_{j_{2r}}\}$ is positively (counterclockwise)
ordered, that is, triods of arcs in consecutive pieces are
positive. We say that it \emph{orientates $p$} (or \emph{it is an
orientation for $p$}) if it contains all rigid pieces. Pairs
$E_{j_l},E_{j_{l+r}}$, $1\leq j\leq r$, are said to be
\emph{opposed} for this orientation. Observe that, since $A$ is
very simple, $p$ always admit an orientation (unless there are no
rigid pieces and at most one flexible leaf with respect to $p$).
We denote by $P_A^*$ the set of nodes from $P_A$ admitting an
orientation. We remark that if an endpoint of a sprig $E$ belongs
to $P_A$, then it also belongs to $P_A^*$.

Similarly, if $E\in \mathcal{E}_A$, let $v_1,\ldots v_l$ be the
points from $P_A\cap \Bd E$ (note that this set may be empty if
$E$ is a sprig). We say that $v_m$ is \emph{rigid (with respect to
$E$)} if, not counting $E$ itself, all pieces containing $v_m$ are
rigid or bland with respect to $v_m$, and the number of rigid
pieces is odd. We say that $v_m$ is \emph{flexible} if there is
some piece, different from $E$, which is flexible with respect to
$v_m$. If $v_m$ is neither rigid nor flexible, then it is called
\emph{bland}. The definitions of \emph{orientation} and that of
\emph{opposed pairs}, are analogous to the previous ones. Except
again in the case when there are no rigid nodes in $\Bd E$ and at
most one flexible node, or $E$ is a sprig and some of its
endpoints is not in $P_A$, $E$ always admits an orientation. In
particular, if $E$ is a sprig and both endpoints $v_1$ and $v_2$
of $E$ belong to $P_A$, then $\{v_1,v_2\}$ is an orientation (the
only possible one). The set of pieces from $\mathcal{E}_A$
admitting an orientation will be denoted by $\mathcal{E}_A^*$.
Figure~\ref{figura3} exhibits some examples of the previously
defined notions.

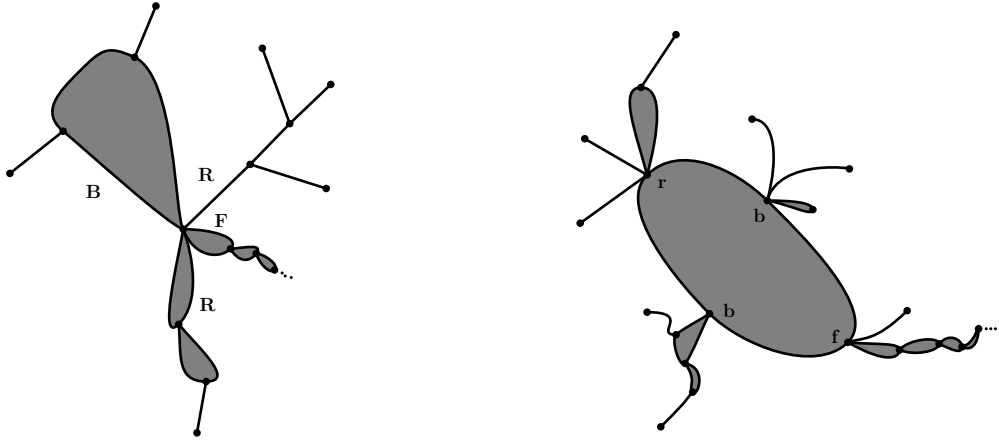
\begin{figure}
\centering
\begin{tikzpicture}[scale=2.0]

\draw[line width=1] (1.29,{3.2-1.70}) -- (1.73,{3.2-1.27});
\draw[line width=1] (1.73,{3.2-1.27}) -- (1.99,{3.2-1.0});
\draw[line width=1] (1.99,{3.2-1.0}) -- (2.26,{3.2-0.74});
\draw[line width=1] (1.81,{3.2-0.5}) -- (1.99,{3.2-1.00});
\draw[line width=1] (2.23,{3.2-1.43}) -- (1.73,{3.2-1.27});

\node[scale=0.6] at (1.44,{3.2-1.34}) {\textbf{R}};

\node[scale=0.6] at (1.54,{3.2-1.64}) {\textbf{F}};

\draw[line width=1] (0.97,{3.2-0.56}) -- (1.11,{3.2-0.22});
\draw[line width=1] (0.5,{3.2-1.05}) -- (0.15,{3.2-1.33});

 \filldraw[line width=1, fill=gray]
     (0.97,{3.2-0.56}) .. controls +(155:0.2) and +(45:0.2) ..  (0.61,{3.2-0.65})
                        .. controls +(225:0.2) and +(140:0.2) ..  (0.5,{3.2-1.05})
                        .. controls +(-40:0.2) and +(165:0.1) ..  (1.29,{3.2-1.70})
                        .. controls +(125:0.1) and +(-25:0.3) ..  (0.97,{3.2-0.56});

    \node[scale=0.6] at (0.7,{3.2-1.45}) {\textbf{B}};

  \filldraw[line width=1, fill=gray]
     (1.29,{3.2-1.70}) .. controls +(260:0.2) and +(225:0.2) ..  (1.26,{3.2-2.33})
                        .. controls +(45:0.2) and +(290:0.2) ..  (1.29,{3.2-1.70});

    \filldraw[line width=1, fill=gray]
     (1.26,{3.2-2.33}) .. controls +(275:0.2) and +(180:0.2) ..  (1.44,{3.2-2.71})
                        .. controls +(0:0.2) and +(-50:0.2) ..  (1.26,{3.2-2.33});

\draw[line width=1] (1.44,{3.2-2.71}) -- (1.38,{3.2-3.05});

\node[scale=0.6] at (1.45,{3.2-2.2}) {\textbf{R}};

\filldraw[line width=1, fill=gray]
     (1.29,{3.2-1.70}) .. controls +(290:0.2) and +(225:0.1) ..  (1.60,{3.2-1.83})
                        .. controls +(45:0.1) and +(0:0.2) ..  (1.29,{3.2-1.70});

\filldraw[line width=1, fill=gray]
     (1.60,{3.2-1.83}) .. controls +(300:0.1) and +(235:0.08) ..  (1.77,{3.2-1.86})
                        .. controls +(55:0.08) and +(5:0.1) ..  (1.60,{3.2-1.83});

\filldraw[line width=1, fill=gray]
     (1.77,{3.2-1.86}) .. controls +(280:0.05) and +(225:0.05) ..  (1.89,{3.2-1.97})
                        .. controls +(45:0.05) and +(-5:0.05) ..  (1.77,{3.2-1.86});

\fill (0.97,{3.2-0.56}) circle (0.025); \fill (0.5,{3.2-1.05})
circle (0.025); \fill (1.29,{3.2-1.70}) circle (0.025); \fill
(1.26,{3.2-2.33}) circle (0.025); \fill (1.44,{3.2-2.71}) circle
(0.025); \fill (1.38,{3.2-3.05}) circle (0.025); \fill
(1.73,{3.2-1.27}) circle (0.025); \fill (1.99,{3.2-1.0}) circle
(0.025); \fill (2.26,{3.2-0.74}) circle (0.025); \fill
(1.81,{3.2-0.5}) circle (0.025); \fill (2.23,{3.2-1.43})circle
(0.025); \fill (1.11,{3.2-0.22}) circle (0.025); \fill
(0.15,{3.2-1.33}) circle (0.025);

\fill (1.60,{3.2-1.83}) circle (0.025); \fill (1.77,{3.2-1.86})
circle (0.025); \fill (1.89,{3.2-1.97}) circle (0.025); \fill
(1.94,{3.2-1.99}) circle (0.01); \fill (1.96,{3.2-2.01})circle
(0.01); \fill (2.00,{3.2-2.02}) circle (0.01);


 \filldraw[line width=1, fill=gray]
     (4.34,{3.2-1.34}).. controls +(225:0.3) and +(135:0.3) ..  (4.75,{3.2-2.26})
                        .. controls +(315:0.3) and +(225:0.3) ..  (5.66,{3.2-2.45})
                        .. controls +(45:0.3) and +(-45:0.3) ..  (5.13,{3.2-1.51})
                        .. controls +(135:0.3) and +(45:0.3) ..  (4.34,{3.2-1.34});

 \filldraw[line width=1, fill=gray]
     (4.34,{3.2-1.34}).. controls +(75:0.3) and +(0:0.15) ..  (4.3,{3.2-0.76})
                        .. controls +(180:0.15) and +(105:0.3) ..  (4.34,{3.2-1.34});

\draw[line width=1] (4.3,{3.2-0.76}) -- (4.53,{3.2-0.41});
\draw[line width=1] (3.9,{3.2-1.66}) -- (4.34,{3.2-1.34});
\draw[line width=1] (3.93,{3.2-1.10}) -- (4.34,{3.2-1.34});

 \filldraw[line width=1, fill=gray]
     (5.13,{3.2-1.51}).. controls +(-10:0.15) and +(270:0.05) ..  (5.43,{3.2-1.57})
                        .. controls +(90:0.05) and +(0:0.15) ..  (5.13,{3.2-1.51});

 \draw[line width=1]
     (5.13,{3.2-1.51}).. controls +(75:0.3) and +(0:0.15) ..  (5.03,{3.2-0.97});

 \draw[line width=1]
     (5.13,{3.2-1.51}).. controls +(75:0.3) and +(0:0.06) ..  (5.67,{3.2-1.30});

 \filldraw[line width=1, fill=gray]
     (4.59,{3.2-2.59}).. controls +(305:0.15) and +(135:0.05) ..  (4.64,{3.2-2.78})
                        .. controls +(315:0.05) and +(-10:0.15) ..  (4.59,{3.2-2.59});

 \filldraw[line width=1, fill=gray]
     (4.75,{3.2-2.26}).. controls +(215:0.05) and +(45:0.05) ..  (4.53,{3.2-2.4})
                        .. controls +(225:0.05) and +(225:0.05) ..  (4.59,{3.2-2.59})
                        .. controls +(45:0.05) and +(240:0.05) ..  (4.75,{3.2-2.26});

 \draw[line width=1]
     (4.34,{3.2-2.25}).. controls +(0:0.3) and +(180:0.15) ..  (4.53,{3.2-2.40});

 \draw[line width=1]
     (4.43,{3.2-3.01}).. controls +(45:0.3) and +(225:0.06) ..  (4.64,{3.2-2.78});

 \draw[line width=1]
     (6.05,{3.2-2.24}).. controls +(225:0.3) and +(20:0.1) ..  (5.66,{3.2-2.45});

 \filldraw[line width=1, fill=gray]
     (5.66,{3.2-2.45}).. controls +(-30:0.3) and +(270:0.05) ..  (6.00,{3.2-2.5})
                        .. controls +(90:0.05) and +(-5:0.2) ..  (5.66,{3.2-2.45});

 \filldraw[line width=1, fill=gray]
     (6.00,{3.2-2.5}).. controls +(-25:0.1) and +(270:0.05) ..  (6.26,{3.2-2.46})
                        .. controls +(90:0.05) and +(50:0.1) ..  (6.00,{3.2-2.5});

 \filldraw[line width=1, fill=gray]
     (6.26,{3.2-2.46}).. controls +(-25:0.1) and +(270:0.05) ..  (6.41,{3.2-2.48})
                        .. controls +(90:0.05) and +(50:0.1) ..  (6.26,{3.2-2.46});

  \filldraw[line width=1, fill=gray]
     (6.41,{3.2-2.48}).. controls +(-25:0.1) and +(270:0.05) ..  (6.52,{3.2-2.36})
                        .. controls +(90:0.05) and +(15:0.1) ..  (6.41,{3.2-2.48});

\fill (6.05,{3.2-2.24}) circle (0.025); \fill (6.00,{3.2-2.5})
circle (0.025); \fill (6.26,{3.2-2.46}) circle (0.025); \fill
(6.41,{3.2-2.48}) circle (0.025); \fill (6.52,{3.2-2.36}) circle
(0.025);

\fill (6.57,{3.2-2.36}) circle (0.01); \fill (6.60,{3.2-2.36})
circle (0.01); \fill (6.63,{3.2-2.36}) circle (0.01);

\fill (4.34,{3.2-1.34}) circle (0.025); \fill (4.75,{3.2-2.26})
circle (0.025); \fill (5.66,{3.2-2.45})circle (0.025); \fill
(5.13,{3.2-1.51}) circle (0.025);

\fill (3.9,{3.2-1.66}) circle (0.025); \fill (3.93,{3.2-1.10})
circle (0.025); \fill (4.3,{3.2-0.76}) circle (0.025); \fill
(4.53,{3.2-0.41}) circle (0.025);

\fill (5.43,{3.2-1.57}) circle (0.025); \fill (5.67,{3.2-1.30})
circle (0.025); \fill (5.03,{3.2-0.97}) circle (0.025);

\fill (4.53,{3.2-2.40}) circle (0.025); \fill (4.34,{3.2-2.25})
circle (0.025); \fill (4.59,{3.2-2.59}) circle (0.025); \fill
(4.64,{3.2-2.78}) circle (0.025); \fill (4.43,{3.2-3.01}) circle
(0.025);

\node[scale=0.6] at (4.88,{3.2-2.24}) {\textbf{b}};
\node[scale=0.6] at (4.44,{3.2-1.4}) {\textbf{r}};
\node[scale=0.6] at (5.08,{3.2-1.61}) {\textbf{b}};
\node[scale=0.6] at (5.58,{3.2-2.41}) {\textbf{f}};

\end{tikzpicture}

\caption{\label{figura3} In the positive (counterclockwise) sense:
a node with a flexible (F) leaf, a rigid (R) sprig, a bland (B)
leaf and a rigid leaf (left), and a leaf with flexible (f), bland
(b), rigid (r) and bland nodes (right).}
\end{figure}

We say that a node $v\in P_A^*$ and a piece $E\in \mathcal{E}_A^*$
are \emph{linked} if $v\in \Bd E$. A sequence
$(\gamma_n)_{n=n_0}^{n_1}$, $-\infty\leq n_0<n_1\leq \infty$, of
pairwise distinct elements from $P_A^*$ and $\mathcal{E}_A^*$ is
called a \emph{chain} if, for any $n_0\leq n<n_1$, the pair
$\gamma_n,\gamma_{n+1}$ is linked. Hence, a chain can be finite,
infinite at one side, or infinite at both sides, and nodes and
pieces (its \emph{links}) alternate at it.

Let $\mathbf{F}=(\mathbf{f}_n)_{n=n_0}^{n_1}$ be a sequence of
orientations for the elements of a chain
$\Gamma=(\gamma_n)_{n=n_0}^{n_1}$. We say that $\Gamma$ is
\emph{oriented by} $\mathcal{F}$ if $\gamma_{n_0+1}\in
\mathbf{f}_{n_0}$, $\gamma_{n_1-1}\in \mathbf{f}_{n_1}$ (when
$n_0$ and/or $n_1$ are finite), and $\gamma_{n-1}$ and
$\gamma_{n+1}$ belong to and are opposed for the orientation
$\mathbf{f}_n$ for any $n_0<n<n_1$. Note that the last condition
only makes sense when $\Gamma$ has at least three links.
Therefore, a two-links chain is oriented if just
$\gamma_{n_1}=\gamma_{n_0+1}\in \mathbf{f}_{n_0}$ and
$\gamma_{n_0}=\gamma_{n_1-1}\in \mathbf{f}_{n_1}$.

Assume that $\Gamma$ is oriented by $\mathbf{F}$. We say that
$\Gamma$ is \emph{complete} if either $n_1=\infty$ (respectively,
$n_0=-\infty$), or $\gamma_{n_1}$ (respectively, $\gamma_{n_0}$)
is a node and the opposed piece to $\gamma_{n_1-1}$ (respectively,
$\gamma_{n_0+1}$) for $\mathbf{f}_{n_1}$ (respectively,
$\mathbf{f}_{n_0}$) is a sprig $L$ whose other endpoint does not
belong to $P_A$ (such a sprig $L$ is called a \emph{stub} of the
complete oriented chain $\Gamma$).  Observe that, by definition, a
complete oriented chain has at least three links.

\begin{lemma}
 \label{linked}
Let $A$ be a very simple shrub. Let $\gamma$ belong to $P_A^*$ or
$\mathcal{E}_A^*$, let $\mathbf{f}$ be an orientation for
$\gamma$, and take $\gamma'\in \mathbf{f}$. If $\gamma$ is a node
and $\gamma'$ is a sprig, assume additionally that the other
endpoint of $\gamma'$ belongs to $P_A$. Then $\gamma'$ belongs to
$\mathcal{E}_A^*$ or $P_A^*$, according to the case, and there is
an orientation $\mathbf{f}'$ for $\gamma'$ containing $\gamma$. In
other words, $(\gamma,\gamma')$ is oriented by
$(\mathbf{f},\mathbf{f'})$.
\end{lemma}

\begin{proof}
Assume that $v=\gamma$ is a node. Then $E=\gamma'$ is a flexible
or rigid piece with respect to $v$. Also, observe that $v$ is
flexible or rigid with respect to $E$. If $E$ is a sprig, then, by
hypothesis, its other endpoint $w$ belongs to $P_A$, $E$ belongs
to $\mathcal{E}_A^*$ and $\mathbf{f}'=\{v,w\}$ does the work. So
we can assume that $E$ is a leaf. If $E\notin \mathcal{E}_A^*$,
then $v$ must be flexible and the other nodes in $E$ must be bland
with respect to $E$, which implies the contradiction that $E$ is
bland with respect to $v$. Hence $E\in \mathcal{E}_A^*$. Now, the
only way to prevent that $v$ belongs to an orientation for $E$ is
that there is an even number of rigid nodes  and $v$ is the only
flexible node with respect to $E$. But this, again, implies the
contradiction that $E$ is bland with respect to $v$.

If $E=\gamma$ is a piece, then $v=\gamma'$ is flexible o rigid
with respect to $E$, $E$ is flexible or rigid with respect to $v$
and we can reason similarly as in the previous paragraph.
\end{proof}

The next result follows immediately from Lemma~\ref{linked}.

\begin{lemma}
 \label{finished}
Let $A$ be a very simple shrub. Let $\gamma$ belong to $P_A^*$ or
$\mathcal{E}_A^*$, let $\mathbf{f}$ be an orientation for $\gamma$
and assume that $\gamma'$ is opposed to $\gamma''$ for
$\mathbf{f}$. Then one of the following alternatives hold:
 \begin{itemize}
 \item[(i)] $\gamma$ is a node, and both $\gamma'$ and
$\gamma''$ are sprigs whose other endpoints do not belong to
$P_A$;
  \item[(ii)] $\gamma$ is a node,  $\gamma''$ (but not $\gamma'$) is
a sprig whose other endpoint do not belong to $P_A$,  and there is
a complete oriented chain $\Gamma=(\gamma_n)_{n=0}^{n_1}$ such
that $\gamma_0=\gamma$, $\mathbf{f}_0=\mathbf{f}$ is the
corresponding orientation for $\gamma$, $\gamma_1=\gamma'$, and
$\gamma''$ is a stub of $\Gamma$;
 \item[(iii)] there is a complete oriented
 chain $(\gamma_n)_{n=n_0}^{n_1}$, with $n_0<0<n_1$, such that
$\gamma_0=\gamma$, $\mathbf{f}_0=\mathbf{f}$ is  the corresponding
orientation for $\gamma$, $\gamma_{-1}=\gamma''$, and
$\gamma_1=\gamma'$.
\end{itemize}
\end{lemma}

When in the statement of the lemma below we speak about a
``complete oriented chain'', we mean a chain that, after using for
its links the orientations from $\mathbf{F}$, becomes oriented and
complete.

\begin{lemma}
 \label{goodchoice}
Let $A$ be a very simple shrub. Then there is a family
$\mathbf{F}$  of  orientations for all nodes from $P_A^*$ and all
pieces from $\mathcal{E}_A^*$ such that, for each sprig $L$ of
$A$, one of the following alternatives hold:
 \begin{itemize}
 \item[(i)] no endpoint of $L$ belongs to $P_A$;
 \item[(ii)] $L$ is a stub of a complete oriented chain;
 \item[(iii)] $L$ is a link of a complete oriented chain.
 \end{itemize}
\end{lemma}

\begin{proof}
Consider the following equivalence relation in $P_A^*\cup
\mathcal{E}_A^*$: $\gamma\sim \gamma'$ if and only if
$\gamma=\gamma'$ or there is a chain containing both $\gamma$ and
$\gamma'$. We just need to explain how to orientate all members
from a specific equivalence class, as the method can be
independently applied to each class. The construction clearly
implies, as we will see, that the conditions in the statement of
the lemma are satisfied.

If a given class consists of just one member, then, in view of
Lemma~\ref{finished}, this unique member $v$ is a node and there
are $2r$ pieces linked to it, all of which are sprigs whose other
points do not belong to $P_A$. Then, we choose for $v$ the only
possible orientation: that consisting of these $2r$ sprigs.

Assume now that our equivalence class contains at least two
elements $\gamma$ and $\gamma''$, and connect them via a chain
$(\gamma_n)_{n=0}^{n_1}$, $\gamma_0=\gamma$,
$\gamma_{n_1}=\gamma''$. To orientate the elements of this chain
(and possibly some other members from this equivalence class) we
proceed as follows. Let $\mathbf{f}$ be an arbitrary orientation
for $\gamma$, take $\gamma'\in \mathbf{f}$ in $P_A^*\cup
\mathcal{E}_A^*$ and apply Lemma~\ref{finished} to obtain a
complete oriented chain: the corresponding orientations of the
links of this chain, in particular of $\gamma$, are those we use
in $\mathbf{F}$. Say that the last link of
$(\gamma_n)_{n=0}^{n_1}$ in this complete chain is $\gamma_{n_0}$,
$0\leq n_0\leq n_1$. If $n_0=n_1$, then we have already finished.
Otherwise we orientate $\gamma_{n_0+1}$, and possibly some other
members of the equivalence class, as follows:

(a) If  $\gamma_{n_0+1}$ belongs to the orientation $\mathbf{f}'$
of  $\gamma_{n_0}$, then we apply again Lemma~\ref{finished} (with
$\gamma_{n_0}$, $\mathbf{f}'$ and $\gamma_{n_0+1}$ playing the
role  of $\gamma$, $\mathbf{f}$ and $\gamma'$ there), to get a new
complete oriented chain, and use the corresponding orientations
for all the links of the chain, in particular for all $\gamma_n$
for some $n_0+1\leq n_0'$ and $n_0+1\leq n\leq n_0'$.

(b) If $\gamma_{n_0+1}$ does not belong to $\mathbf{f}'$, then let
$(\alpha_m)_{m=0}^{m_1}$, $\alpha_0=\gamma_{n_0}$,
$\alpha_1=\gamma_{n_0+1}$, $m_1\leq \infty$, be a maximal chain
with the property that each possible orientation for $\alpha_m$
contains $\alpha_{m-1}$, $1\leq m\leq m_1$. (Here we allow the
degenerate case $m_1=0$,  meaning that there is an orientation for
$\gamma_{n_0+1}$ not containing $\gamma_{n_0}$.) Realize that, in
fact, there is just one possible orientation
$\mathbf{g}_m=\{\alpha_{m-1},\alpha_{m+1}\}$ for any $1\leq m\leq
m_1$, when $\alpha_{m_1+1}$ has the property, in the finite case,
that it admits an orientation $\mathbf{g}_{m_1+1}$ not containing
$\alpha_{m_1}$. Observe that none of the links $\alpha_m$ can be a
sprig, as this would imply that $\gamma_{n_0+1}$ is rigid with
respect to $\mathbf{f}'$, and then $\gamma_{n_0+1}\in
\mathbf{f}'$. We add all these orientations $\mathbf{g}_m$
(including $m=m_1+1$ in the finite case) to $\mathbf{F}$, in
particular orienting all $\gamma_n$ for some $n_0+1\leq n_0'$ and
$n_0+1\leq n\leq n_0'$.

If $n_0'=n_1$, then we stop; otherwise, we keep applying (a) or
(b) as before to orientate all links $\gamma_n$, and possibly some
other members (maybe all) of the equivalence class. If some
$\beta''$ from the class remains to be oriented, we connect it to
an already oriented member $\beta$ via a chain
$(\beta_l)_{l=0}^{l_1}$, $\beta_0=\beta$, $\beta_{l_1}=\beta''$,
with all the links of this chain to be oriented except $\beta$,
and proceed as previously explained applying (a) or (b) according
to the case.
\end{proof}

We are ready to prove the main result of this subsection. Note
that it already implies Theorem~\ref{teoB}, via
Theorem~\ref{analytic-1} and Proposition~\ref{realizing}, whenever
the shrub $A$ is simple.

\begin{proposition}
\label{finiterealizable} Let $A$  be a simple shrub and let $T$
contain all odd buds of $A$ and exactly one point from every odd
cactus of $A$. Then there exist a homeomorphism $h:\EE \to \EE$
and an analytic map $F: \EE\setminus h(T) \to \R$ whose set of
zeros is contained in $h(A\setminus T)$ and contains $h((\Bd
A)\setminus T)$.
\end{proposition}

\begin{proof}

Firstly, we assume that $A$ is very simple and, after the
identification $\EE\cong\R_\infty^2$ and without loss of
generality, that $A\subset \RR$. If $A$ has no odd buds, then $A$
is either a singleton, and the proposition is trivial, or a cactus
(this is a simple consequence of \cite[Theorem~7.1, p.
64]{HARARY}). In the latter case, it is easy to construct a cactus
$A'\subset \R_\infty^2$, compatible with $A$ (in the sense of
Theorem~\ref{exthomeo}), with one of its leaves being the disk
$\{z\in \RR:|z|\geq 1\}\cup\{\infty\}$, all other leaves being
hypocycloids included in the unit disk. Then there is an analytic
map in $\R_\infty^2$ whose set of zeros is $\Bd A'\cup\{\infty\}$
(Lemma~\ref{analytic-hypo} and Remark~\ref{algebraic}), and the
proposition follows.

Therefore, we assume in what follows that $A$ has some odd bud
$w$.  Find $\mathbf{F}$ as in Lemma~\ref{goodchoice} and fix a
skeleton  of $A$ which we assume (this is the most difficult case)
to be infinite: call it $\Psi=(S_k)_{k=1}^\infty$. It is not
restrictive to suppose that $w$ is one of the endpoints of $S_1$.
It is not difficult to construct an skeleton
$\Psi'=(S_k')_{k=1}^\infty$, with each thick arc $S_k'$ generated
from the corresponding $S_k$ similarly as $A'$ is generated from
$A$ in Lemma~\ref{hypo-segment} (and then linearly compressed as
explained in Remark~\ref{rem-hypo-segment} to resemble ``almost
segments'', and appropriately translated and rotated), so that, on
the one hand, $\Psi'$ is extensible, hence the closure $A'$ of the
union set of the thick arcs $S_k'$ is a shrub (use
Proposition~\ref{skeleton1} for this, taking also
Remark~\ref{skeleton0} into account), and, on the other hand,
$\Psi$ and $\Psi'$ are compatible, hence $A$ and $A'$ are
compatible as well (Proposition~\ref{skeleton2}). Thus, in
particular, all leaves of $A'$ are hypocycloids, and all sprigs of
$A'$ are segments. This can be done regardless the sets of nodes,
the families of disks and the assignations, call them $Q_k$,
$\mathcal{D}_k$ and $\phi_k$, we use when applying
Lemma~\ref{hypo-segment} to the stems $S_k$, and additionally
assuming that all non-cusp points of the leaves of $A'$  are
exterior points of $A'$. Now  we specify $Q_k$, $\mathcal{D}_k$
and $\phi_k$ as follows. The points $Q_k$ (respectively, the disks
$\mathcal{D}_k$) are among those in $S_k$ belonging to,
respectively, $P_A^*$ and $\mathcal{E}_A^*$. More precisely, a
node $v$ belongs to $Q_k$ if the number of pieces of
$\mathbf{f}_v$ from $E$ to $E'$ (not counting them), in the
positive sense, is different from the number of pieces in the
negative sense: here $\mathbf{f}_v$ is the orientation in
$\mathbf{F}$ corresponding to $v$, and $E$ and $E'$ are the pieces
in $S_k$ linked to $v$, $E$ being closer to $w$ than $E'$. Note
that $E$ and/or $E'$ may belong, or not, to $\mathbf{f}_v$. If the
first number is larger (respectively, smaller) than the second
one, then we put $\phi_k(v)=1$ (respectively, $\phi_k(v)=-1$).
Analogously, a leaf $D$ belongs to $\mathcal{D}_k$ if the number
of nodes of $\mathbf{f}_D$ (the orientation in $\mathbf{F}$
corresponding to $D$) from $v$ to $v'$ (the nodes in $S_k$ linked
to $D$), in the positive sense, is different from the number of
nodes in the negative sense, and write $\phi_k(D)=\pm 1$ according
to the case.

Next revise, if necessary, the construction of $\Psi'$ to ensure:
(a) opposed nodes for any $\mathbf{f}_D$ go, via the compatibility
homeomorphism, to opposed cusps in the corresponding hypocycloid
of $A'$ (the second paragraph of Remark~\ref{rem-hypo-segment} is
useful in this regard); (b) opposed pieces for any $\mathbf{f}_v$
go to aligned segments and/or hypocycloids in $A'$. Clearly, this
is possible due to the way the assignations $\phi_k$ have been
made. Consider the family of segments consisting of all sprigs of
$A'$ and the diameters of the hypocycloids in $\mathcal{E}_{A'}^*$
whose endpoints, when homeomorphically seen in $A$, are opposed
for some orientation $\mathbf{f}_D$ in $\mathbf{F}$. Note that
some segments from this family may be aligned, and then pasted
into larger segments. We do so (and take closures in the case when
we are pasting infinitely many segments) to receive a family of
maximal segments $\mathcal{S}'$ of $A'$. The properties of
$\mathbf{F}$ guarantee that the endpoints of all segments from
$\mathcal{S}'$ are odd buds.

We are almost done (in the very simple case): if $T'$ denotes the
set of odd buds of $A'$, then Lemmas~\ref{analytic-segment} and
\ref{analytic-hypo}, together with Theorem~\ref{bruhat}, imply
that $\Bd A'$ is analytic in $\RR\setminus T'$ (although maybe not
in $\R_\infty^2\setminus T'$). With a little extra care, we can
get $w'$ (the corresponding point to $w$ in $A'$) to be the point
$(-\pi/2,0)$, and  moreover $A'\setminus \{w'\}\subset
(-\pi/2,\pi/2)^2$. Let $f(x,y)=(\tan(x),\tan(y))$. Clearly $A'$
(and then $A$) and $A''=\{\infty\}\cup f(A'\setminus \{w'\})$ are
compatible, and $\Bd A''$ is analytic in $\R_\infty^2\setminus
T''$, $T''$ being the set of odd buds of $A''$. From this,
Proposition~\ref{finiterealizable} follows.

If $B$ (which we assume again to be a plane set) is simple, but
not very simple, and $T_0\subset T$ contains exactly one point
from each odd cactus of $B$, we add, for each $p\in T_0$, an arc
$B_p$ to $B$ which intersect $B$ exactly at $p$ (if $p$ is either
a node or an exterior point of $B$), or at an exterior point
belonging to the same leaf as $p$ (if $p$ is an interior point of
$B$). Certainly we can assume, and so we do, that these arcs are
small enough so that $A=B\cup \bigcup_{p\in T_0} B_p$ is a (very
simple) shrub. Construct a plane homeomorphism $g$ mapping $A$
onto a very simple shrub $A'$ with ``optimal analyticity'', as
previously explained, when there is no loss of generality in
assuming that the interior points from $T_0$ are mapped by $g$ to
the centers of some hypocycloids in $A'$, and realize that
$g(T_0)$ consists of points belonging to (but not being endpoints
of) some segments from the family $\mathcal{S}'$. After cutting
these segments off so that the points $g(T_0)$ become endpoints of
the resultant segments (and then analyticity is lost exactly at
them), we recover $B$ (via $g^{-1}$) and are done.
\end{proof}

\subsection{The general case}
 \label{general-case}

Let $A$ be a shrub and let $D$ be a leaf of $A$. We say that $D$
is \emph{special} if it contains an odd bud of $A$. Also, we say
that an arc $L\subset \Bd D$ is \emph{special} if one of its
endpoints (the \emph{distinguished} endpoint of $L$) is an odd bud
of $A$, there is no other leaf containing the second endpoint of
$L$, and all other points of $L$ are exterior points of $A$. We
say that $A$ is \emph{semi-simple} if any special leaf of $A$
includes some special arc.

Assume now that $D$ is \emph{very special}, which means that it is
special and its boundary contains some star point of odd order of
$\Bd A$.  Then we can associate to each such point $s_m$ an arc
$S_m\subset D$ connecting $s_m$ to a ``close'' odd bud in $\Bd D$.
If this is carefully done (for instance, deforming $D$ to a ball,
using a segment to connect the star point to its closest odd bud
in the circumference, and going back to $D$), different arcs
$S_m,S_l$ intersect at most at its common endpoint (an odd bud),
and, when we delete from  $A$ the interior points of $D$, and add
the arcs $S_m$, the resultant set is a Peano space (although not a
shrub because simply connectedness is lost). We call such a
(countable) family $\{S_m\}_m$ a \emph{web (for $D$)}. If the arcs
$J_m$ are strictly contained in $S_m$ but share with $S_m$ an
endpoint $q_m$ (an odd bud), then we call $\{J_m\}_m$ a
\emph{cutting} of the web $\{S_m\}_m$ with \emph{distinguished}
endpoints $\{q_m\}_m$.

The following lemma is the last ingredient we need to complete the
proof of Theorem~\ref{teoB}. Note that there may be indexes $i$
for which the cutting $\{J_{i,m}'\}_m$ is empty: this means that
$D_i'$ is special, but not very special.

\begin{lemma}
 \label{semi-simple}
Let $A$ be a shrub. Then there are a semi-simple  shrub $A'$ and a
continuous map $g:\EE\to \EE$ such that $g(A')=A$. Moreover, for
every special leaf $D_i'$ of $A'$ there are a special arc
$L_i'\subset \Bd D_i'$ and a cutting $\{J_{i,m}'\}_m$ such that
$g$ is constant on each $L_i'$ and $J_{i,m}'$, and maps
bijectively $\EE$ onto $\EE$ otherwise.
\end{lemma}

\begin{proof}

Firstly, we assume that $A$ has no very special leaves. Then the
proof of the lemma is based on the following

\medbreak\noindent\emph{Claim.} Let $A$ be a shrub having a
special leaf $D$ and let $p\in\Bd D$ be an odd bud of $A$. Let $U$
be a neighbourhood of $p$. Then there are a shrub $B$ and a
continuous map $h:\EE\to\EE$ satisfying the following conditions:
\begin{itemize}
\item[(i)] $D$ is a special leaf of $B$ whose boundary
 includes a special arc $L$ having $p$ as its
 distinguished point;
\item[(ii)] $L\subset U$;
 \item[(iii)] $h(B)=A$, $h(D)=D$,
 $h(L)=\{p\}$, $h$ maps bijectively
 $(\EE\setminus L)\cup \{p\}$ onto $\EE$, and $h$ is the
 identity map outside $U$ and on each leaf $D'$
 of $B$ (different from $D$) containing $p$.
\end{itemize}

Let $R=\EE\setminus A$. Recall that $\Bd A$ is a Peano space.
Then, by the Carath\'eodory theorem \cite[Theorem~9.8 and
Lemma~9.8, pp. 279-280]{Pommerenke}, there is a continuous  map
$f:\mathbb{D}^2\to\Cl R$ mapping the interior of the unit ball
$\mathbb{D}^2$ homeomorphically onto $R$, and the unit circle
$\mathbb{S}^1$ onto $\Bd R=\Bd A$; moreover, if $q\in \Bd A$ does
not disconnects $\Bd A$, then it has exactly one preimage under
$f$. There are many such ``non-disconnecting'' points in $\Bd D$:
in fact, except for a countable set of points of $\Bd D$ (those
belonging to some stem intersecting $\Bd D$ exactly at one point;
take also the skeleton structure of $A$ into account), all of them
are of this type. In particular, we can find a sequence
$(q_n)_{n=1}^\infty$ in $\Bd D$ of non-disconnecting points
converging monotonically to $p$ and, with the help of $f$, a
sequence of arcs $Q_n$ with endpoints $q_n$ and $q_{n+1}$ and
included otherwise in $R$, with $\diam(Q_n)\to 0$, so that if $E$
is the disk encircled by $Q_n$ and the arc in $\Bd D$ with
endpoints $q_n$ and $q_{n+1}$ satisfying $E\cap \Inte
D=\emptyset$, then $E$ does not intersect any other leaf
containing $p$. Then, modifying slightly $\bigcup_{n=1}^\infty
Q_n$, we find an arc $Q$, with endpoints $p$ and $q_1$, and
included in $R$ otherwise, so that if $Q'$ is the arc in $\Bd D$
with endpoints $p$ and $q_1$ and containing the other points
$q_n$, and $E'$ is the disk encircled by $Q\cup Q'$ satisfying
$E\cap\Inte D=\emptyset$, then $E'$ does not intersect any other
leaf containing $p$.

Now, after identifying $\EE$ with  $\R_\infty^2$, and applying
Theorem~\ref{exthomeo}, there is no loss of generality in assuming
that $p$ is the origin, $D$ is the rectangle $[0,1]\times [-1,0]$,
$Q$ is the polygonal $\Bd ([0,1]^2)\setminus \{(x,0):0<x<1\}$ and
the square $[0,1]^2$ does not intersect the other leaves of $A$
containing $p$ (except at $p$). Find an arc with endpoints
$(1/2,1)$ and $p$, and included in $(0,1)^2$ otherwise, such that
the region enclosed by it, the horizontal segment $[0,1/2]\times
\{1\}$ and the vertical segment $\{0\}\times [0,1]$ does not
intersects $A$ (except at $p$). Again due to
Theorem~\ref{exthomeo}, we can assume that this arc is in fact the
segment joining $p$ and $(1/2,1)$.

Find $0<\epsilon<1$ such that $[0,\epsilon]\times
[-\epsilon,\epsilon]\subset U$. Let $h$ be defined as the identity
outside this rectangle and mapping affinely the segments
$[0,\epsilon/2]\times \{y\}$ and $[\epsilon/2,\epsilon]\times
\{y\}$ onto $[0,|y|/2]\times \{y\}$ and $[|y|/2,\epsilon]\times
\{y\}$, respectively, for any $y\in [-\epsilon,\epsilon]$. Also,
let $B=h^{-1}(A)$ and $L=\{(x,0):0\leq x\leq \epsilon/2\}$.
Clearly, $B$, $h$ and $L$ satisfy the conditions (i), (ii) and
(iii) and the claim is proved.

\medbreak

Now, to prove the lemma (recall that we are assuming the
non-existence of very special leaves), consider the most difficult
case when $A$ has infinitely many special leaves
$\{D_i\}_{i=1}^\infty$ and fix  odd buds $p_i\in D_i$ (note that
it is possible $p_i=p_k$ for some $i\neq k$). Successive
applications of the claim allow us to find shrubs $A_i$, disks
$D_i'$ and $U_i'$, arcs $L_i'$, points $p_i'$ and continuous maps
$g_i:\EE\to \EE$ (when, if $i_1\leq i_2$, then we denote
$g_{i_1,i_2}=g_{i_1}\circ\cdots \circ g_{i_2}$), such that, for
any $k$, the following is true:
\begin{itemize}
\item[(a)] all disks $\{D_i'\}_{i=1}^k$ are special leaves for
 $A_k$  having the arcs $L_i'$ as corresponding special arcs with
 distinguished endpoints $p_i'$; moreover $D_k'$ and $p_k'$ are
 also, respectively, a special leaf and an odd bud of $A_{k-1}$
 (here we mean $A_0=A$).
\item[(b)] $\diam(U_k')\leq 1/2^k$ and $L_k'\subset U_k'$;
moreover,
 if $i<k$, then either
 $U_k'\subset  U_i'$ or $U_k'\cap U_i'=\emptyset$;
\item[(c)] $g_k(A_k)=A_{k-1}$, $g_k(D_k')=D_k'$,
 $g_k(L_k')=\{p_k'\}$, $g_{1,k-1}(D_k')=D_k$, $g_{1,k-1}(p_k')=p_k$
 (here $g_{1,0}$ denotes the identity map),
 $g_k$ maps bijectively $(\EE\setminus
 L_k')\cup \{p_k'\}$ onto $\EE$, and $g_k$ is the identity map outside $U_k'$
 and on each of the disks $D_i'$, $1\leq i<k$.
\end{itemize}

By (b) and (c) the sequences $(g_{k,i})_{i=k}^\infty$ converge
uniformly to  onto continuous maps $g_k^*$. We next show that
$g=g_1^*$ and $A'=g^{-1}(A)$ (together with the sets $D_i'$ and
$L_i'$) satisfy the requirements of the lemma.

We begin by proving the last statement of the lemma (that
concerning bijectivity).  Clearly, $g_k^*=g_{k,l-1}\circ g_l^*$
for any $k<l$. Write $M'=\bigcup_{i=1}^\infty L_i'$ as a countable
union  of a family of pairwise disjoint sets $\{C_j'\}_j$. By (a)
and (b), they are dendrites and their diameters tend to zero if
there are infinitely many of them; moreover, the arcs $L_i'$ in
each of these dendrites share a unique point  $r_j'$ (we are using
that non-distinguished endpoints of special arcs just belong to
one leaf). By (c), $g$ is constant on each dendrite $C_j'$. Hence,
to prove the statement, we are left to show that if $u\neq v$ and
$g(u)=g(v)$, then there is $C_j'$ such that $u,v\in C_j'$.

Again by (c), $g$ is the identity outside the union of the sets
$U_i'$ (so $u,v \in \bigcup_i U_i'$). Three possibilities arise.
If $u$ and $v$ belongs to the same $U_j'$ for infinitely many $j$,
then (b) implies that $u=v$. If $u,v\in U_k'$  but neither $u$ nor
$v$ belong to some set $U_i'$ included in $U_k'$, then $g_{k+1}^*$
is the identity on $u$ and $v$, and hence $g_{1,k}(u)=g_{1,k}(v)$,
so some $C_j'$ contains both $u$ and $v$. Finally, if none of
these two previous cases hold, there must exist $U_k'$ such that,
say, $u\in U_k'$, $v\in \EE\setminus U_k'$. Furthermore,  (b) and
(c) together imply that $g_k^*$ maps $\EE\setminus U_k'$ onto
itself and $U_k'$ onto itself. Hence, after writing
$u^*=g_k^*(u)$, $v^*=g_k^*(v)$, we conclude that $u^*\neq v^*$ but
$g_{1,k-1}(u^*)=g_{1,k-1}(v^*)$. This is only possible, by (c), if
there are arcs $L_{i_1}',L_{i_2}'$ with $1\leq i_1\leq i_2\leq
k-1$ and $p_{i_1}'=p_{i_2}'$ such that both $u^*$ and $v^*$ belong
to $C=L_{i_1}'\cup L_{i_2}'$. But, due to (c), $g_k$ is the
identity on $C$ and maps $\EE\setminus C$ onto itself, so
$g_k^*(u)=g_k(g_{k+1}^*(u))=u^*$ implies $g_{k+1}^*(u)=u^*$, and
in general we have that $g_i^*(u)=u^*$ for all $i\geq k$. Since,
by (b), the maps $g_i^*$ converge uniformly to the identity, we
obtain $u=u^*$ and similarly $v=v^*$. If $C_j'$ is the dendrite in
$M$ including $C$, we conclude $u,v\in C_j'$ as we desired to
prove.

A consequence of the previous result is that $g^{-1}(u)$ is a
connected set for any $u\in \EE$, which by \cite[Theorem 9, p.
131]{KU} implies that if $C\subset \EE$ is connected, then
$g^{-1}(C)$ (and similarly all sets $(g_k^*)^{-1}(C)$) is
connected as well. Now the rest of statements of the lemma are
easy to prove. For instance, $A'$ is connected, hence a continuum,
locally connected, hence a Peano space (because if $u'\in A'$,
$\epsilon>0$ is given,  $k$ is large enough and $V$ is a connected
neighbourhood of $g_k^*(u')$ in the shrub $A_k$ of diameter less
than $\epsilon$, then $(g_k^*)^{-1}(V)$ is a connected
neighbourhood of $u'$ in $A'$ of diameter less than $2\epsilon$),
and $\EE\setminus A'=g^{-1}(\EE\setminus A)$ is connected, hence
$A$ is a shrub. Also, $g(D_i')=D_i$ and $g(p_i')=p_i$ for any $i$,
which, together with the bijectivity property, guarantees that all
disks $D_i'$ are leaves of $A'$ and all points $p_i'$ are odd buds
of $A'$. In fact, since each leaf of $A'$ is mapped onto a leaf of
$A$, the disks $D_i'$ are precisely the special leaves of $A'$.
Using that $g_{i+1}^*(A')=A_{i+1}$ and $g_{i+1}^*$ is the identity
on $D_i'$, which is a special leaf of $A_{i+1}$ with special arc
$L_i'$, it is simple to show that $L_i'$  is also special for
$D_i'$ in $A'$. In particular, $A'$ is semi-simple. This concludes
the proof the lemma in the case when there are no cuttings.

In the general case, we first generate an intermediate semi-simple
shrub $A'$ and a surjective map mapping $A'$ onto $A$, which we
now call $\tilde{g}$, as just explained. Then we construct webs
for all very special leaves of $A'$ (which are the leaves mapped
by $\tilde{g}$ to the very special leaves of $A$), and apply
similar (but simpler) ideas to those already explained to collapse
small cuttings of these webs to their distinguished endpoints via
a surjective continuous map $\tilde{g}':\EE\to\EE$ which is the
identity outside the very special leaves of $A'$, so
$\tilde{g}'(A')=A'$. Then $g=\tilde{g}'\circ \tilde{g}$ is the map
we are looking for.
\end{proof}

Let $A$ be a shrub and let $T\subset A$ contain all odd buds of
$A$ and one point from every odd cactus of $A$. Find a map $g$ and
a semi-simple shrub $A'$ as in Lemma~\ref{semi-simple}, and then
construct a simple shrub $A''$ by first removing from $A'$ all
special arcs $L_i'$ (except its endpoints) and the interior points
of all special leaves, and then adding the arcs
$\Cl(S_{i,m}'\setminus J_{i,m}')$, with $\{S_{i,m}'\}_m$ being the
webs from which we have obtained the cuttings $\{J_{i,m}'\}_m$.
Let $T''=g^{-1}(T)\cap A''$. Clearly, $T''$ contains all odd buds
of $A''$ and one point from every odd cactus of $A''$.

Now, by Proposition~\ref{finiterealizable}, there are a
homeomorphism $f:\EE\to \EE$, a simple  shrub $A^*$, and an
analytic map $F:\EE\setminus T^*\to \R$ (here $T^*=f^{-1}(T'')$),
such that $f(A^*)=A''$, $F(u)\neq 0$ for any $u\in \EE\setminus
A^*$, and $F(u)=0$ for any $u\in (\Bd A^*)\setminus T^*$. Let
$L_i^*=f^{-1}(L_i')$ and $J_{i,m}^*=f^{-1}(J_{i,m}')$ for any $i$
and $m$. Write $M^*=\bigcup_i (L_i^*\cup \bigcup_m J_{i,m}^*)$ as
the countable union of some pairwise disjoint dendrites
$\{C_j^*\}_j$ (with their diameters tending to zero in the
infinite case). Then $f^*=g\circ f$ is constant on each dendrite
$C_j^*$, and if we distinguish a point $r_j^*\in C_j^*$ for each
$j$, then we get that $f^*$ maps bijectively $(\EE\setminus
M^*)\cup \{r_j^*\}_j$ onto $\EE$. See Figure~\ref{figura4}.

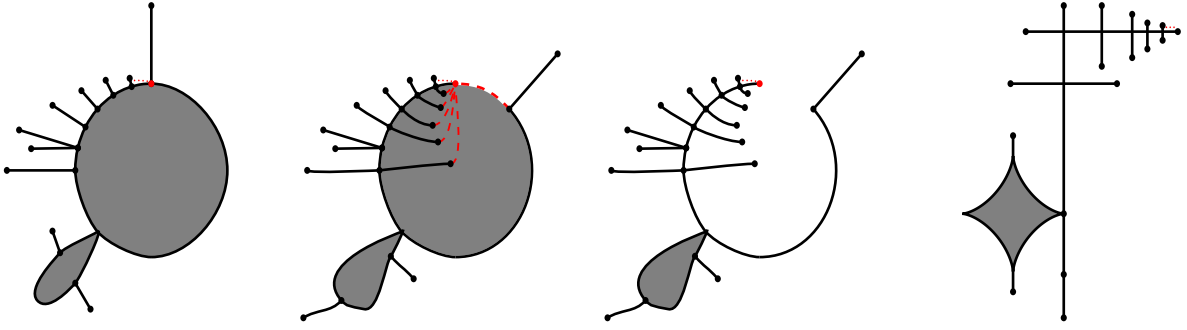
\begin{figure}
    \centering
\begin{tikzpicture}[scale=1.0, yscale=1.15]


 \filldraw[line width=1, fill=gray]
     (2,-1) .. controls +(0:0.5) and +(270:0.5) ..  (3,0)
                .. controls +(90:0.5) and +(0:0.5) .. (2,1)
                        .. controls +(180:0.1) and +(15:0.1) .. ({2+cos(90 + 15)},{sin(90 + 15)})
                        .. controls +(195:0.1) and +(30:0.1) ..  ({2+cos(90 + 30)},{sin(90 + 30)})
                        .. controls +(210:0.1) and +(45:0.1) ..  ({2+cos(90 + 45)},{sin(90 + 45)})
                        .. controls +(225:0.1) and +(60:0.1) ..  ({2+cos(90 + 60)},{sin(90 + 60)})
                        .. controls +(240:0.1) and +(75:0.1) ..  ({2+cos(90 + 75)},{sin(90 + 75)})
                        .. controls +(255:0.1) and +(90:0.1) ..  (1,0)
                        .. controls +(270:0.2) and +(135:0.2) ..  (1.293,-0.707)
                        .. controls +(315:0.2) and +(180:0.2) ..  (2,-1);

 \filldraw[line width=1, fill=gray]
     (1.293,-0.707) .. controls +(210:0.1) and +(45:0.1) ..  (0.8,-0.95)
                        .. controls +(225:0.1) and +(135:0.2) ..  (0.5,-1.5)
                        .. controls +(315:0.2) and +(225:0.1) ..  (1.0,-1.3)
                .. controls +(45:0.1) and +(30:0.1) .. (1.293,-0.707);

\draw[line width=1] (0.8,-0.95) -- (0.7,-0.7); \fill (0.8,-0.95)
circle (0.04); \fill (0.7,-0.7) circle (0.04); \draw[line width=1]
(1.0,-1.3) -- (1.2,-1.6); \fill (1.0,-1.3) circle (0.04); \fill
(1.2,-1.6) circle (0.04);

\draw[line width=1] ({2+cos(90 + 15)},{sin(90 + 15)}) -- ({2+1.1*cos(90 + 15)},{1.1*sin(90 + 15)}); 
\fill ({2+cos(90 + 15)},{sin(90 + 15)}) circle (0.04); \fill
({2+1.1*cos(90 + 15)},{1.1*sin(90 + 15)}) circle (0.04);

\draw[line width=1] ({2+cos(90 + 30)},{sin(90 + 30)}) -- ({2+1.2*cos(90 + 30)},{1.2*sin(90 + 30)}); 
\fill ({2+cos(90 + 30)},{sin(90 + 30)}) circle (0.04); \fill
({2+1.2*cos(90 + 30)},{1.2*sin(90 + 30)}) circle (0.04);

\draw[line width=1] ({2+cos(90 + 45)},{sin(90 + 45)}) -- ({2+1.3*cos(90 + 45)},{1.3*sin(90 + 45)}); 
\fill ({2+cos(90 + 45)},{sin(90 + 45)}) circle (0.04); \fill
({2+1.3*cos(90 + 45)},{1.3*sin(90 + 45)}) circle (0.04);

\draw[line width=1] ({2+cos(90 + 60)},{sin(90 + 60)}) -- ({2+1.5*cos(90 + 60)},{1.5*sin(90 + 60)}); 
\fill ({2+cos(90 + 60)},{sin(90 + 60)}) circle (0.04); \fill
({2+1.5*cos(90 + 60)},{1.5*sin(90 + 60)}) circle (0.04);

\draw[line width=1] ({2+cos(90 + 75)},{sin(90 + 75)}) -- ({2+1.8*cos(90 + 75)},{1.8*sin(90 + 75)}); 
\fill ({2+cos(90 + 75)},{sin(90 + 75)}) circle (0.04); \fill
({2+1.8*cos(90 + 75)},{1.8*sin(90 + 75)}) circle (0.04);
\draw[line width=1] ({2+cos(90 + 75)},{sin(90 + 75)}) -- ({2+1.6*cos(90 + 81)},{1.6*sin(90 + 81)}); 
\fill ({2+cos(90 + 75)},{sin(90 + 75)}) circle (0.04); \fill
({2+1.6*cos(90 + 81)},{1.6*sin(90 + 81)}) circle (0.04);

\draw[line width=1] ({2+cos(180)},{sin(180)}) -- ({2+1.9*cos(180)},{1.9*sin(180)}); 
\fill ({2+cos(180)},{sin(180)}) circle (0.04); \fill
({2+1.9*cos(180)},{1.9*sin(180)}) circle (0.04);

\draw[line width=1] ({2+cos(90)},{sin(90)}) -- ({2+1.9*cos(90)},{1.9*sin(90)}); 
\fill ({2+cos(90)},{sin(90)}) circle (0.04); \fill
({2+1.9*cos(90)},{1.9*sin(90)}) circle (0.04);

\fill[color=red] ({2+1.03*cos(93)},{1.03*sin(93)}) circle (0.01);
\fill[color=red] ({2+1.04*cos(96)},{1.04*sin(96)}) circle (0.01);
\fill[color=red] ({2+1.05*cos(99)},{1.05*sin(99)}) circle (0.01);
\fill[color=red] ({2+1.06*cos(102)},{1.06*sin(102)}) circle
(0.01);

\fill[color=red] (2.0,1.0) circle (0.04);


\begin{scope}[xshift=4cm]

 \filldraw[line width=1, color=white, fill=gray]
     (2,-1) .. controls +(0:0.9) and +(315:0.9) ..  ({2+cos(45)},{sin(45)})
                .. controls +(135:0.3) and +(0:0.3) .. (2,1)
                        .. controls +(180:0.1) and +(15:0.1) .. ({2+cos(90 + 15)},{sin(90 + 15)})
                        .. controls +(195:0.1) and +(30:0.1) ..  ({2+cos(90 + 30)},{sin(90 + 30)})
                        .. controls +(210:0.1) and +(45:0.1) ..  ({2+cos(90 + 45)},{sin(90 + 45)})
                        .. controls +(225:0.1) and +(60:0.1) ..  ({2+cos(90 + 60)},{sin(90 + 60)})
                        .. controls +(240:0.1) and +(75:0.1) ..  ({2+cos(90 + 75)},{sin(90 + 75)})
                        .. controls +(255:0.1) and +(90:0.1) ..  (1,0)
                        .. controls +(270:0.2) and +(135:0.2) ..  (1.293,-0.707)
                        .. controls +(315:0.2) and +(180:0.2) ..  (2,-1);

 \draw[line width=1]
     (2,-1) .. controls +(0:0.9) and +(315:0.9) ..  ({2+cos(45)},{sin(45)});

\draw[line width=1, color=red, dashed]
     ({2+cos(45)},{sin(45)})  .. controls +(135:0.3) and +(0:0.3) .. (2,1);

\draw[line width=1]
     (2,1)  .. controls +(180:0.1) and +(15:0.1) .. ({2+cos(90 + 15)},{sin(90 + 15)})
                        .. controls +(195:0.1) and +(30:0.1) ..  ({2+cos(90 + 30)},{sin(90 + 30)})
                        .. controls +(210:0.1) and +(45:0.1) ..  ({2+cos(90 + 45)},{sin(90 + 45)})
                        .. controls +(225:0.1) and +(60:0.1) ..  ({2+cos(90 + 60)},{sin(90 + 60)})
                        .. controls +(240:0.1) and +(75:0.1) ..  ({2+cos(90 + 75)},{sin(90 + 75)})
                        .. controls +(255:0.1) and +(90:0.1) ..  (1,0)
                        .. controls +(270:0.2) and +(135:0.2) ..  (1.293,-0.707)
                        .. controls +(315:0.2) and +(180:0.2) ..  (2,-1);

 \filldraw[line width=1, fill=gray]
     (1.293,-0.707) .. controls +(200:0.1) and +(135:0.6) ..  (0.5,-1.5)
                        .. controls +(315:0.1) and +(170:0.2) ..  (0.8,-1.6)
                        .. controls +(350:0.2) and +(225:0.1) ..  (1.15,-0.99)
                .. controls +(55:0.1) and +(50:0.1) .. (1.293,-0.707);


 \draw[line width=1]
     (0.5,-1.5) .. controls +(225:0.2) and +(30:0.2) ..  (0.0,-1.7);
\fill (0.5,-1.5) circle (0.04); \fill  (0.0,-1.7) circle (0.04);

 \draw[line width=1]
     (1.15,-0.99) .. controls +(315:0.2) and +(135:0.2) ..  (1.45,-1.25);

\fill (1.15,-0.99) circle (0.04); \fill (1.45,-1.25) circle
(0.04);

\draw[line width=1]
     ({2+1.1*cos(90 + 15)},{1.1*sin(90 + 15)}) .. controls +(285:0.05) and +(105:0.05) .. ({2+cos(90 + 15)},{sin(90 + 15)})
                                                   .. controls +(285:0.05) and +(180:0.05) .. ({2+0.9*cos(90 + 10)},{0.9*sin(90 + 10)});
\draw[line width=0.7, color=red, dashed]
     ({2+0.9*cos(90 + 10)},{0.9*sin(90 + 10)})  .. controls +(0:0.05) and +(230:0.05) .. (2.0,1.0);

\fill ({2+0.9*cos(90 + 10)},{0.9*sin(90 + 10)}) circle (0.04);
\fill ({2+1.1*cos(90 + 15)},{1.1*sin(90 + 15)}) circle (0.04);
\fill ({2+cos(90 + 15)},{sin(90 + 15)}) circle (0.04);

\draw[line width=1]
     ({2+1.2*cos(90 + 30)},{1.2*sin(90 + 30)}) .. controls +(285:0.05) and +(120:0.05) .. ({2+cos(90 + 30)},{sin(90 + 30)})
                                                   .. controls +(300:0.05) and +(180:0.05) .. ({2+0.75*cos(90 + 15)},{0.75*sin(90 + 15)});
\draw[line width=0.7, color=red, dashed]
     ({2+0.75*cos(90 + 15)},{0.75*sin(90 + 15)})  .. controls +(0:0.05) and +(240:0.05) .. (2.0,1.0);

\fill ({2+0.75*cos(90 + 15)},{0.75*sin(90 + 15)}) circle (0.04);
\fill ({2+cos(90 + 30)},{sin(90 + 30)}) circle (0.04); \fill
({2+1.2*cos(90 + 30)},{1.2*sin(90 + 30)}) circle (0.04);

\draw[line width=1]
     ({2+1.3*cos(90 + 45)},{1.3*sin(90 + 45)}) .. controls +(290:0.05) and +(140:0.05) .. ({2+cos(90 + 45)},{sin(90 + 45)})
                                                   .. controls +(320:0.05) and +(180:0.2) .. ({2+0.6*cos(90 + 30)},{0.6*sin(90 + 30)});
\draw[line width=0.7, color=red, dashed]
     ({2+0.6*cos(90 + 30)},{0.6*sin(90 + 30)})  .. controls +(0:0.2) and +(250:0.05) .. (2.0,1.0);

\fill ({2+0.6*cos(90 + 30)},{0.6*sin(90 + 30)}) circle (0.04);
\fill ({2+cos(90 + 45)},{sin(90 + 45)}) circle (0.04); \fill
({2+1.3*cos(90 + 45)},{1.3*sin(90 + 45)}) circle (0.04);

\draw[line width=1]
     ({2+1.5*cos(90 + 60)},{1.5*sin(90 + 60)}) .. controls +(290:0.05) and +(140:0.05) .. ({2+cos(90 + 60)},{sin(90 + 60)})
                                                   .. controls +(320:0.05) and +(180:0.2) .. ({2+0.4*cos(90 + 35)},{0.4*sin(90 + 35)});
\draw[line width=0.7, color=red, dashed]
     ({2+0.4*cos(90 + 35)},{0.4*sin(90 + 35)})  .. controls +(0:0.2) and +(250:0.05) .. (2.0,1.0);

\fill ({2+0.4*cos(90 + 35)},{0.4*sin(90 + 35)}) circle (0.04);
\fill ({2+cos(90 + 60)},{sin(90 + 60)}) circle (0.04); \fill
({2+1.5*cos(90 + 60)},{1.5*sin(90 + 60)}) circle (0.04);

\draw[line width=1] ({2+cos(90 + 75)},{sin(90 + 75)}) -- ({2+1.8*cos(90 + 75)},{1.8*sin(90 + 75)}); 
\fill ({2+cos(90 + 75)},{sin(90 + 75)}) circle (0.04); \fill
({2+1.8*cos(90 + 75)},{1.8*sin(90 + 75)}) circle (0.04);
\draw[line width=1] ({2+cos(90 + 75)},{sin(90 + 75)}) -- ({2+1.6*cos(90 + 81)},{1.6*sin(90 + 81)}); 
\fill ({2+cos(90 + 75)},{sin(90 + 75)}) circle (0.04); \fill
({2+1.6*cos(90 + 81)},{1.6*sin(90 + 81)}) circle (0.04);

\draw[line width=1]
     ({2+1.95*cos(180)},{1.95*sin(180)}) .. controls +(310:0.05) and +(180:0.05) .. ({2+cos(180)},{sin(180)})
                                                   .. controls +(5:0.05) and +(180:0.2) .. ({2+0.1*cos(90 + 40)},{0.1*sin(90 + 40)});
\draw[line width=0.7, color=red, dashed]
     ({2+0.1*cos(90 + 40)},{0.1*sin(90 + 40)})  .. controls +(0:0.2) and +(250:0.05) .. (2.0,1.0);

\fill ({2+0.1*cos(90 + 40)},{0.1*sin(90 + 40)}) circle (0.04);
\fill ({2+cos(180)},{sin(180)}) circle (0.04); \fill
({2+1.95*cos(180)},{1.95*sin(180)}) circle (0.04);

%

\draw[line width=1] ({2+cos(45)},{sin(45)}) -- ({2+1.9*cos(45)},{1.9*sin(45)}); 
\fill ({2+cos(45)},{sin(45)}) circle (0.04); \fill
({2+1.9*cos(45)},{1.9*sin(45)}) circle (0.04);

\fill[color=red] (2.0,1.0) circle (0.04);

\fill[color=red] ({2+1.03*cos(93)},{1.03*sin(93)}) circle (0.01);
\fill[color=red] ({2+1.04*cos(96)},{1.04*sin(96)}) circle (0.01);
\fill[color=red] ({2+1.05*cos(99)},{1.05*sin(99)}) circle (0.01);
\fill[color=red] ({2+1.06*cos(102)},{1.06*sin(102)}) circle
(0.01);
\end{scope}

\begin{scope}[xshift=8cm]


 \draw[line width=1]
     (2,-1) .. controls +(0:0.9) and +(315:0.9) ..  ({2+cos(45)},{sin(45)});


\draw[line width=1]
     (2,1)  .. controls +(180:0.1) and +(15:0.1) .. ({2+cos(90 + 15)},{sin(90 + 15)})
                        .. controls +(195:0.1) and +(30:0.1) ..  ({2+cos(90 + 30)},{sin(90 + 30)})
                        .. controls +(210:0.1) and +(45:0.1) ..  ({2+cos(90 + 45)},{sin(90 + 45)})
                        .. controls +(225:0.1) and +(60:0.1) ..  ({2+cos(90 + 60)},{sin(90 + 60)})
                        .. controls +(240:0.1) and +(75:0.1) ..  ({2+cos(90 + 75)},{sin(90 + 75)})
                        .. controls +(255:0.1) and +(90:0.1) ..  (1,0)
                        .. controls +(270:0.2) and +(135:0.2) ..  (1.293,-0.707)
                        .. controls +(315:0.2) and +(180:0.2) ..  (2,-1);

 \filldraw[line width=1, fill=gray]
     (1.293,-0.707) .. controls +(200:0.1) and +(135:0.6) ..  (0.5,-1.5)
                        .. controls +(315:0.1) and +(170:0.2) ..  (0.8,-1.6)
                        .. controls +(350:0.2) and +(225:0.1) ..  (1.15,-0.99)
                .. controls +(55:0.1) and +(50:0.1) .. (1.293,-0.707);


 \draw[line width=1]
     (0.5,-1.5) .. controls +(225:0.2) and +(30:0.2) ..  (0.0,-1.7);
\fill (0.5,-1.5) circle (0.04); \fill  (0.0,-1.7) circle (0.04);

 \draw[line width=1]
     (1.15,-0.99) .. controls +(315:0.2) and +(135:0.2) ..  (1.45,-1.25);

\fill (1.15,-0.99) circle (0.04); \fill (1.45,-1.25) circle
(0.04);

\draw[line width=1]
     ({2+1.1*cos(90 + 15)},{1.1*sin(90 + 15)}) .. controls +(285:0.05) and +(105:0.05) .. ({2+cos(90 + 15)},{sin(90 + 15)})
                                                   .. controls +(285:0.05) and +(180:0.05) .. ({2+0.9*cos(90 + 10)},{0.9*sin(90 + 10)});

\fill ({2+0.9*cos(90 + 10)},{0.9*sin(90 + 10)}) circle (0.04);
\fill ({2+1.1*cos(90 + 15)},{1.1*sin(90 + 15)}) circle (0.04);
\fill ({2+cos(90 + 15)},{sin(90 + 15)}) circle (0.04);

\draw[line width=1]
     ({2+1.2*cos(90 + 30)},{1.2*sin(90 + 30)}) .. controls +(285:0.05) and +(120:0.05) .. ({2+cos(90 + 30)},{sin(90 + 30)})
                                                   .. controls +(300:0.05) and +(180:0.05) .. ({2+0.75*cos(90 + 15)},{0.75*sin(90 + 15)});

\fill ({2+0.75*cos(90 + 15)},{0.75*sin(90 + 15)}) circle (0.04);
\fill ({2+cos(90 + 30)},{sin(90 + 30)}) circle (0.04); \fill
({2+1.2*cos(90 + 30)},{1.2*sin(90 + 30)}) circle (0.04);

\draw[line width=1]
     ({2+1.3*cos(90 + 45)},{1.3*sin(90 + 45)}) .. controls +(290:0.05) and +(140:0.05) .. ({2+cos(90 + 45)},{sin(90 + 45)})
                                                   .. controls +(320:0.05) and +(180:0.2) .. ({2+0.6*cos(90 + 30)},{0.6*sin(90 + 30)});

\fill ({2+0.6*cos(90 + 30)},{0.6*sin(90 + 30)}) circle (0.04);
\fill ({2+cos(90 + 45)},{sin(90 + 45)}) circle (0.04); \fill
({2+1.3*cos(90 + 45)},{1.3*sin(90 + 45)}) circle (0.04);

\draw[line width=1]
     ({2+1.5*cos(90 + 60)},{1.5*sin(90 + 60)}) .. controls +(290:0.05) and +(140:0.05) .. ({2+cos(90 + 60)},{sin(90 + 60)})
                                                   .. controls +(320:0.05) and +(180:0.2) .. ({2+0.4*cos(90 + 35)},{0.4*sin(90 + 35)});

\fill ({2+0.4*cos(90 + 35)},{0.4*sin(90 + 35)}) circle (0.04);
\fill ({2+cos(90 + 60)},{sin(90 + 60)}) circle (0.04); \fill
({2+1.5*cos(90 + 60)},{1.5*sin(90 + 60)}) circle (0.04);

\draw[line width=1] ({2+cos(90 + 75)},{sin(90 + 75)}) -- ({2+1.8*cos(90 + 75)},{1.8*sin(90 + 75)}); 
\fill ({2+cos(90 + 75)},{sin(90 + 75)}) circle (0.04); \fill
({2+1.8*cos(90 + 75)},{1.8*sin(90 + 75)}) circle (0.04);
\draw[line width=1] ({2+cos(90 + 75)},{sin(90 + 75)}) -- ({2+1.6*cos(90 + 81)},{1.6*sin(90 + 81)}); 
\fill ({2+cos(90 + 75)},{sin(90 + 75)}) circle (0.04); \fill
({2+1.6*cos(90 + 81)},{1.6*sin(90 + 81)}) circle (0.04);

\draw[line width=1]
     ({2+1.95*cos(180)},{1.95*sin(180)}) .. controls +(310:0.05) and +(180:0.05) .. ({2+cos(180)},{sin(180)})
                                                   .. controls +(5:0.05) and +(180:0.2) .. ({2+0.1*cos(90 + 40)},{0.1*sin(90 + 40)});

\fill ({2+0.1*cos(90 + 40)},{0.1*sin(90 + 40)}) circle (0.04);
\fill ({2+cos(180)},{sin(180)}) circle (0.04); \fill
({2+1.95*cos(180)},{1.95*sin(180)}) circle (0.04);

%

\draw[line width=1] ({2+cos(45)},{sin(45)}) -- ({2+1.9*cos(45)},{1.9*sin(45)}); 
\fill ({2+cos(45)},{sin(45)}) circle (0.04); \fill
({2+1.9*cos(45)},{1.9*sin(45)}) circle (0.04);

\fill[color=red] (2.0,1.0) circle (0.04);

\fill[color=red] ({2+1.03*cos(93)},{1.03*sin(93)}) circle (0.01);
\fill[color=red] ({2+1.04*cos(96)},{1.04*sin(96)}) circle (0.01);
\fill[color=red] ({2+1.05*cos(99)},{1.05*sin(99)}) circle (0.01);
\fill[color=red] ({2+1.06*cos(102)},{1.06*sin(102)}) circle
(0.01);
\end{scope}

 \draw[line width=1] (14,1.9) -- (14,-1.7);
\fill (14,1.9) circle (0.04); \fill (14,-1.7) circle (0.04);

\fill (14,-1.2) circle (0.04);

\fill (14,-0.5) circle (0.04);

\draw[line width=1] (13.3333, 0.4) -- (13.3333, -1.4); \fill
(13.3333, 0.4) circle (0.04); \fill (13.3333, -1.4) circle (0.04);

\filldraw[line width=1, fill=gray,
 smooth,samples=100,domain=0:360]
 plot(
 {13 + 1/3 + 1/6*(3*cos(\x)+cos(3*\x))},
 {-1/2+1/6*(3*sin(\x)-sin(3*\x))}
 );


\draw[line width=1] (13.3, 1.0) -- (14.7, 1.0); \fill (13.3, 1.0)
circle (0.04); \fill (14.7, 1.0) circle (0.04);

\draw[line width=1] (13.5, 1.6) -- (15.5, 1.6); \fill (13.5, 1.6)
circle (0.04); \fill (15.5, 1.6) circle (0.04);

\draw[line width=1] (14.5, 1.9) -- (14.5, 1.2); \fill (14.5, 1.9)
circle (0.04); \fill (14.5, 1.2) circle (0.04);

\draw[line width=1] (14.9, 1.8) -- (14.9, 1.3); \fill (14.9, 1.8)
circle (0.04); \fill (14.9, 1.3) circle (0.04);

\draw[line width=1] (15.1, 1.7) -- (15.1, 1.4); \fill (15.1, 1.7)
circle (0.04); \fill (15.1, 1.4) circle (0.04);

\draw[line width=1] (15.3, 1.67) -- (15.3, 1.5); \fill (15.3,
1.67) circle (0.04); \fill (15.3, 1.5) circle (0.04);

\fill[color=red] (15.35, 1.65) circle (0.01); \fill[color=red]
(15.40, 1.65) circle (0.01); \fill[color=red] (15.45, 1.65) circle
(0.01);

\end{tikzpicture}\caption{\label{figura4} From left to right, the
shrubs $A$, $A'$, $A''$ and $A^*$.}
\end{figure}

By Proposition~\ref{equiv}(ii), if we define the equivalence
relation $\sim$ in $\EE$ by $u\sim v$ if either $u=v$ or there is
$j$ such that both $u$ and $v$ belong to $C_j^*$, then the
quotient space $\Sigma=\EE/\sim$ is homeomorphic to $\EE$. Endow
$\Sigma$ with an analytic differential structure and let
$\Pi:\EE\to\Sigma$ be the projection map. Let $K^*=M^*\cup T^*$,
$U^*=\EE\setminus K^*$, and $\mathcal{V}=\Pi(U^*)$, and  use
Theorem~\ref{analytic-0} to find a continuous onto map
$\tilde{\Pi}:\EE\to \Sigma$ satisfying $\tilde{\Pi}(u)=\Pi(u)$ for
any $u\in K^*$ and mapping, analytically and diffeomorphically,
$U^*$ onto $\mathcal{V}$. Observe that the map $\Gamma:\Sigma\to
\EE$ mapping each $C_j^*$ (when seen as a point from $\Sigma$) to
$f^*(r_k)$, and each $u\in \EE\setminus M^*$ (again, when seen as
a point from $\Sigma$) to $\tilde{\Pi}^{-1}(f^*(u))$ is a
bijection satisfying $\Gamma\circ \tilde{\Pi}=f^*$. In fact,
$\Gamma$ is continuous (hence a homeomorphism), because if $U$ is
open in $\EE$, then
$\tilde{\Pi}^{-1}(\Gamma^{-1}(U))=(f^*)^{-1}(U)$ is open, hence
$\Gamma^{-1}(U)=\tilde{\Pi}((f^*)^{-1}(U))=\Sigma\setminus
\tilde{\Pi}(\EE\setminus(f^*)^{-1}(U))$ is open as well.

Use again Theorem~\ref{analytic-0} to find an analytic
diffeomorphism $\Phi:\Sigma\to \EE$ and consider the homeomorphism
$h=\Phi\circ \Gamma^{-1}$, the shrub $B=h(A)$ and the set
$P=h(T)$. Then $\Phi \circ \tilde{\Pi}$ maps, analytically and
diffeomorphically, $U^*$ onto $W=\EE\setminus P$. Moreover, $H=
F|_{U^*}\circ (\Phi\circ \tilde{\Pi})|_W^{-1}$ is an analytic map
on $W$ satisfying $H(v)\neq 0$ for any $v\in \EE\setminus B$ and
$H(v)=0$ for any $v\in (\Bd B)\setminus P$. After applying
Theorem~\ref{analytic-1} and Proposition~\ref{realizing},
Theorem~\ref{teoB} follows.

\section*{Acknowledgements}

This work has been partially supported by Ministerio de
Econom\'{\i}a y Competitividad, Spain, grant MTM2014-52920-P. The
first author is also supported by Fundaci\'on S\'eneca by means of
the program ``Contratos Predoctorales de Formaci\'on del Personal
Investigador'', grant 18910/FPI/13.

\bigskip
\noindent \emph{J. G. Esp\'{\i}n Buend\'{\i}a and V. Jim\'enez
L\'opez's address:} {\sc Departamento de Matem\'a\-ticas,
Universidad de Murcia, Campus de Espinardo, 30100 Murcia, Spain.}

e-mails: {\tt josegines.espin@um.es, vjimenez@um.es}

\end{document}